\documentclass[12pt,reqno]{amsart}
\title{Orthogonal root numbers of tempered parameters}
\author{David Schwein}
\date{5 July 2021}
\email{schwein@umich.edu}
\usepackage[margin=1in]{geometry}

\usepackage{microtype}

\usepackage[final,notref,notcite]{showkeys}

\usepackage{tikz-cd}

\usepackage%
[bookmarksdepth=subsection,
unicode]%
{hyperref}

\usepackage{amsmath,amssymb,amsthm}

\usepackage{cleveref} %% cleveref must be loaded after hyperref!

\newcounter{result}
	\newtheorem{conjecture}[result]{Conjecture}
	\newtheorem{corollary}[result]{Corollary}
	\newtheorem{lemma}[result]{Lemma}
	\newtheorem{theorem}[result]{Theorem}
	\newtheorem{theoremx}{Theorem} %% "letter-numbered" theorems
		
	\newtheorem{lemmax}[theoremx]{Lemma} %% "letter-numbered" lemmas

\theoremstyle{definition}
	\newtheorem{definition}[result]{Definition}
	\newtheorem{example}[result]{Example}
	\newtheorem{remark}[result]{Remark}

	\renewcommand\C{\mathbb C}
	\newcommand\F{\mathbb F}
	\renewcommand\G{\mathbb G}
	\newcommand\N{\mathbb N}
	\newcommand\Q{\mathbb Q}
	\newcommand\R{\mathbb R}
	\newcommand\Z{\mathbb Z}

	\newcommand\cal\mathcal
	\renewcommand\frak\mathfrak
	\newcommand\tn\textnormal
	\newcommand\ul\underline

	\DeclareMathOperator\Ad{Ad}
	\DeclareMathOperator\Aut{Aut}
	\DeclareMathOperator\Cl{Cl}
	\DeclareMathOperator\fieldchar{char}
	\DeclareMathOperator\coker{coker}
	\DeclareMathOperator\Hom{Hom}
	\DeclareMathOperator\Ind{Ind}
	\DeclareMathOperator\Lie{Lie}
	\DeclareMathOperator\ord{ord}
	\DeclareMathOperator\sgn{sgn}
	\DeclareMathOperator\WD{WD}

	\DeclareMathOperator\GL{GL}
	\DeclareMathOperator\GPin{Pin}
	\DeclareMathOperator\Pin{Pin}
	\DeclareMathOperator\SO{SO}
	\DeclareMathOperator\Spin{Spin}
	\DeclareMathOperator\SL{SL}

	\newcommand\defeq{\stackrel{\tn{def}}{=}}
	\newcommand\into{\hookrightarrow}
	\newcommand\xto{\xrightarrow}

	\newcommand\rep{{r}}
	\newcommand\HC{1}

\begin{document}
\begin{abstract}
We show that an orthogonal root number
of a tempered $L$-parameter~$\varphi$
decomposes as the product of two other numbers:
the orthogonal root number of the principal parameter
and the value on a certain involution
of Langlands's central character for~$\varphi$.
The formula resolves a conjecture of Gross and Reeder
and computes root numbers of Weil-Deligne representations
arising in the work of Hiraga, Ichino, and Ikeda
on the Plancherel measure.
\end{abstract}

\maketitle
\tableofcontents

\section{Introduction}
To every representation of the absolute Galois group~$\Gamma\!_k$
of a local field~$k$, or more generally,
of its Weil-Deligne group $\WD_k$,
we can attach various local constants,
the $\gamma$-, $L$-, and $\varepsilon$-factors,
each a function of a complex parameter~$s$.
Of these, the $\varepsilon$-factor
has the simplest description
as a function of~$s$:
it is an exponential function,
$s\mapsto a\cdot b^{s-1/2}$.
Understanding this local constant
thus amounts to understanding
the base~$b$ and constant term~$a$.

The base of the $\varepsilon$-factor
(or depending on one's choice of terminology,
its logarithm) is known as the \emph{Artin conductor},
taken as $1$ for $k$ archimedean.
This quantity measures the ramification
of the representation.
Although the Artin conductor has its own subtleties,
especially in the presence of wild ramification,
there are good formulas to compute it
\cite[Chapter~VI]{serre79},
formulas that make precise the sense in which
the conductor measures ramification.

The constant term~$a$ is known as the \emph{root number}.
We denote it here by $\omega$,
an additive function of Weil-Deligne representations.
The choice of exponent $s-1/2$ ensures
that the root number is a complex number of modulus one,
so we can refer to it informally as the sign
of the $\varepsilon$-factor.
This sign is much more subtle than
the Artin conductor.

Here is one example of its subtlety.
The problem of computing a root number
generalizes the problem of computing
the sign of a Gauss sum.
For quadratic Gauss sums,
the sign is determined by a congruence
condition on the modulus,
a classical result of Gauss
and the most difficult of
the basic properties of these sums.
For cubic Gauss sums,
the situation is much more complicated:
there is no congruence condition
to describe the sign,
and in fact, the sign is randomly
and uniformly distributed
on the unit circle
\cite{heath-brown_patterson79}.

In other words, for as simple
a class of Galois representations as
the cubic characters,
the root number exhibits great complexity.
Fortunately, there is a special class
of Galois representations whose root numbers
are more amenable to computation:
the orthogonal representations.
After all, the orthogonal characters
are precisely the quadratic characters.

Using formal properties of root numbers,
it is easy to see that the root number
of an orthogonal representation
is a fourth root of unity,
and the square of the root number
can be described in terms of the determinant
of the representation.
Ultimately, then, computing
an orthogonal root number is a matter
of distinguishing between two square roots.
Deligne calculated the distinction
by expressing the root number
of an orthogonal representation
in terms of the second Stiefel-Whitney class
of the representation \cite{deligne76}.
Stiefel-Whitney classes are a notion from algebraic topology,
but they can be reinterpreted purely in terms
of group cohomology, as the pullback of
a certain universal cohomology class.
In this way, Deligne's formula reduces computing
the root number of an orthogonal representation
to a problem in group cohomology.

In the Langlands program we generalize
the study of Weil-Deligne representations
to the study of $L$-parameters
for a quasi-split reductive $k$-group~$G$.
To extend the definition of the local constants
to this setting we reduce to
the case of the general linear group
by composing an $L$-parameter $\WD_k\to{}^LG$
with a complex representation
$\rep\colon{}^LG\to\GL(V)$
and computing the local constants
of the representation $\rep\circ\varphi$.
Deligne's formula suggests the possibility
of computing the root numbers
$\omega(\varphi,\rep)\defeq\omega(\rep\circ\varphi)$
for orthogonal~$\rep$, in other words,
the \emph{orthogonal root numbers} of~$\varphi$.
The hope is a formula for $\omega(\varphi,\rep)$
that incorporates information about the $L$-parameter~$\varphi$
and the complex orthogonal representation~$\rep$.

In a 2010 paper, the direct inspiration for this article,
Gross and Reeder gave a conjectural formula
for a particular class of orthogonal root numbers
and proved the formula when $G$ is split.
Motivated in part by a conjecture
of Hiraga, Ichino, and Ikeda
on the formal degree of discrete series
\cite[Conjecture~1.4]{hiraga_ichino_ikeda08},
they took the adjoint representation
$\Ad\colon{}^LG\to\GL(\hat{\frak g})$,
an orthogonal representation,
and set about computing $\omega(\varphi,\tn{Ad})$.
Their conjectural answer has three factors.

First, there is a recipe, due to Langlands,
that constructs from the $L$-parameter $\varphi$
a character $\chi_\varphi$ of the center of~$G$.
At the same time, we can
conjecturally assign to $\varphi$
a finite set $\Pi_\varphi$ of 
smooth irreducible representations of~$G(k)$.
The set $\Pi_\varphi$ is called
the \emph{$L$-packet} of~$\varphi$
and the assignment $\varphi\mapsto\Pi_\varphi$
is called the \emph{local Langlands correspondence}.
It is expected that all elements of $\Pi_\varphi$
have the same central character,
and further, that this character is $\chi_\varphi$.

Second, Gross and Reeder evaluated this central character
on a certain \emph{canonical involution} $z_\tn{Ad}$
in the center of~$G$, defined as
the value of the sum of the positive coroots on~$-1$.

Third, when $k$ is nonarchimedean
the $L$-group admits a particular distinguished parameter
called the \emph{principal parameter}
$\varphi_\tn{prin}\colon\WD_k\to{}^LG$.
This parameter is trivial on the Weil group,
and its restriction to the Deligne $\SL_2$ corresponds,
via the Jacobson-Morozov theorem,
to the sum of elements in a pinning of~$\widehat G$,
a nilpotent element of the Lie algebra.
The $L$-packet of the principal parameter
captures the Steinberg representation.
When $k$ is archimedean,
we define the principal parameter
to be the trivial crossed homomorphism
$W_k\to\widehat G$.

\begin{conjecture}[{\cite[Conjecture~8.3]{gross_reeder10}}]
\label{thm2}
If $k$ is nonarchimdean of characteristic zero,
the center of~$G$ is anisotropic,
and the parameter $\varphi\colon\WD_k\to{}^LG$
is discrete then
\[
\frac{\omega(\varphi,\tn{Ad})}%
{\omega(\varphi_\tn{prin},\tn{Ad})}
= \chi_\varphi(z_\tn{Ad}).
\]
\end{conjecture}

The goal of this article is to verify Gross and Reeder's conjecture.
Actually, we will prove a more general version
that relaxes the base field to an arbitrary local field,
the adjoint representation
to an arbitrary orthogonal representation,
and discreteness of the parameter to temperedness.
We generalize Gross and Reeder's
canonical involution $z_\tn{Ad}$ to an involution~$z_\rep$
(see \Cref{sec:cc:inv}) that depends only on
the restriction of~$\rep$ to~$\widehat G$.

The original motivation for our generalization
was to compute the adjoint $\varepsilon$-factor
arising in a conjecture of Hiraga, Ichino, and Ikeda
describing the Plancherel measure
\cite[Conjecture~1.5]{hiraga_ichino_ikeda08},
of which the formal degree conjecture
that motivated Gross and Reeder is a special case.
In the more general conjecture,
the adjoint representation of~${}^LG$ is replaced
by a certain ``relative adjoint representation''
of the $L$-group of a Levi.
From this perspective it is natural,
and in the end costs little,
to relax the adjoint representation
to an arbitrary orthogonal representation.

\begin{theoremx} \label{thm1}
Let $\rep$ be an orthogonal representation of~${}^LG$
and $\varphi\colon\WD_k\to{}^LG$ a tempered parameter.
Then
\[
\frac{\omega(\varphi,\rep)}%
{\omega(\varphi_\tn{prin},\rep)}
= \chi_\varphi(z_\rep).
\]
\end{theoremx}

The \namecref{thm1} does not fully compute
the orthogonal root number $\omega(\varphi,\rep)$:
instead, it disentangles
the root number into an automorphic factor,
$\chi_\varphi(z_\rep)$,
and a Galois factor, $\omega(\varphi_\tn{prin},\rep)$.
If one wanted to use \Cref{thm1}
to pin down an orthogonal root number precisely,
it seems that the main challenge would be to compute
the orthogonal root number of the principal parameter.
In general I expect no better resolution to this problem
than the rough answer provided by Clifford theory,
though for special classes of representations,
such as the adjoint representation
\cite[Equation~(21)]{gross_reeder10},
it might be possible to say more.

Although our work is informed by
the local Langlands correspondence,
everything here takes place on the Galois side.
A stronger version of the theorem would assert
that $\chi_\varphi(z_\rep)$ is actually
the value on~$z_\rep$ of the central character
of the $L$-packet of~$\varphi$;
this is much more difficult
because it requires some knowledge of $L$-packets.
Lapid proved the stronger version of \Cref{thm1}
for generic irreducible representations
of certain classical groups \cite{lapid04}.

Gross and Reeder proved \Cref{thm2} for split~$G$
by an argument in group cohomology.
This article is an outgrowth of an observation
that their argument can be generalized
in various directions.

To relax the split assumption in Gross and Reeder's proof,
we generalize, in \Cref{thm4} of \Cref{sec:gc},
the basic lemma from group cohomology
that underlies their proof,
taking into account the Galois action on~$\widehat G$.
Since the pin extension of a complex orthogonal group
is not topologically split,
we must use Borel cohomology here
instead of continuous (group) cohomology.
With this modification,
\Cref{thm4} proves equal
two particular Borel cohomology classes
in $\tn H_\tn{Borel}^2({}^LG,\{\pm1\})$.
\Cref{thm1} then follows from the equation
by pullback along the parameter~$\varphi$,
once we properly identify the factors of
the pulled-back equation.
Most of this article is devoted to
identifying these factors.

Two of the factors are root numbers,
and their recognition as such passes through Deligne's theorem.
Gross and Reeder already reformulated the theorem
in the language of group cohomology
for determinant-one orthogonal representations of Galois type,
and it is mostly a matter of collecting
definitions from the algebraic topology literature,
in \Cref{sec:sw}, to extend their reformulation
to tempered orthogonal representations.

The third factor of the pulled-back equation
is the value of the central
character~$\chi_\varphi$ on the involution~$z_\rep$.
Identifying this factor requires several detours
that are surely well known to experts in the field.

Any spin representation of a complex reductive group
gives rise to a character of its topological
fundamental group.
The first detour, in \Cref{sec:spin},
is to compute this character.
Here we offer a small correction
to an exercise of Bourbaki.

The second detour, in \Cref{sec:cc}, is to
compute the second cohomology group of the Weil group
with coefficients in the character lattice of a $k$-torus,
generalizing a standard computation for
the absolute Galois group.
It turns out that this cohomology group
is the character group of the Harish-Chandra
subgroup of the torus.
In \Cref{app:karpuk} we extend these results,
for $k$ nonarchimedean of characteristic zero,
to any finite type $k$-group of multiplicative type,
using a generalization of Tate duality
due to Karpuk.

Finally, in the brief \Cref{sec:synth}
we weave together these disparate threads
to prove \Cref{thm1},
and then explain the connection with the conjectures
of Hiraga, Ichino, and Ikeda.

\subsection{Notational conventions}
\label{sec:intro:not}
Let $k$ be a nonarchimedean local field
of residue characteristic~$p$
and characteristic different from~$2$.
The latter assumption is no great loss because
\Cref{thm1} is trivially satisfied if $+1 = -1$.
Let $\Gamma\!_k$ be the absolute Galois group of~$k$,
let $W_k$ be the Weil group of~$k$,
and let 
\[
\WD_k \defeq
\begin{cases}
W_k &\tn{if $k$ is archimedean} \\
W_k\times\SL_2(\C) &\tn{if not}
\end{cases}
\]
be the Weil-Deligne group of~$k$.
In practice we can immediately forget
$\WD_k$ and work with $W_k$ because
root numbers of orthogonal Weil-Deligne
representations are unaffected
by restriction to the Weil group
\cite[Section~2.3]{gross_reeder10}.
For $k$ nonarchimedean, let $I_k\subset W_k$
be the inertia subgroup.

We reserve the letter~$G$ for a group,
of two kinds: either a reductive group
or a topological group.
When $G$ is a topological group,
we assume it to be Hausdorff.
When $G$ is a reductive group,
we assume it to be connected and quasi-split.

To best align with the statement of the key \Cref{thm4}
we work with the Weil form of the $L$-group,
${}^LG\defeq\widehat G\rtimes W_k$.
The choice of the Weil form over the Galois form
is not essential because in practice,
we can replace the Weil-group factor
of this semidirect product with the Galois group
of any extension of~$k$ that is
large enough to contain the splitting field of~$G$
and to trivialize the $L$-parameter
relevant to the problem at hand.

For us, a \emph{representation} of the $L$-group~${}^LG$
is a representation in the sense of Borel's
Corvallis article \cite[(2.6)]{borel_corvallis2},
that is, a finite-dimensional complex vector space~$V$
together with a continuous homomorphism 
${}^LG\to\GL(V)$ whose restriction to
$\widehat G$ is a morphism of complex algebraic groups.
Similarly, a \emph{representation}
of a complex reductive group is assumed algebraic.

In general, the $\varepsilon$-factor, and hence the root number,
depends not only on a Weil-Deligne representation
but also on a nontrivial additive character of~$k$
and a Haar measure on~$k$.
Since there are formulas that explain
the dependence of these factors on the character and the measure,
nothing is lost in fixing them.
We follow Gross and Reeder in computing
$\varepsilon$-factors with respect to
the Haar measure that assigns $\cal O_k$ measure one and
an additive character that is trivial on $\cal O_k$
and nontrivial on the inverse of some uniformizer.

In this article we use three kinds of cohomology:
the singular cohomology of a topological space,
denoted by $\tn H^\bullet_\tn{sing}$;
the continuous cohomology of a topological group,
denoted by $\tn H^\bullet$;
and the Borel cohomology of a topological group,
denoted by $\tn H_\tn{Borel}^\bullet$,
whose definition we review in \Cref{sec:gc:bc}.

Given any group~$A$ of order two,
let $\sgn\colon A\to\{\pm1\}$ denote
the canonical isomorphism
to the group~$\{\pm1\}$ of order two.

\subsection{Acknowledgments}
I am grateful to Jack Carlisle
for discussing the cohomology of classifying spaces,
to Peter Dillery for discussing
the cohomology of the Weil group,
to my advisor, Tasho Kaletha,
for discussing the contents of this article
and providing detailed feedback on it,
and to Karol Koziol for pointing me
to Flach's article \cite{flach08}.

This research was supported by
the National Science Foundation RTG
grant DMS 1840234.

\section{Group cohomology} \label{sec:gc}
This section collects general facts on group cohomology
that inform the rest of this article.
In \Cref{sec:gc:ge}, we review the relationship
between extensions and cohomology for discrete groups.
In \Cref{sec:gc:bc} we explain how the relationship
behaves in the presence of topology.
And in \Cref{sec:gc:kl} we state and prove
the main cohomological lemma, \Cref{thm4},
underlying the proof of \Cref{thm1}.
The remaining sections of the article will
flesh out the connection
between \Cref{thm1} and \Cref{thm4}.

\subsection{Group extensions} \label{sec:gc:ge}
In this subsection we work in the category of discrete groups.
An \emph{extension} of the group~$G$ by the group~$A$
is an exact sequence
\begin{center}
\begin{tikzcd}
1 \rar & A \rar & E \rar & G \rar & 1,
\end{tikzcd}
\end{center}
permitting one to identify $A$ with a subgroup of~$E$
and $G$ with the quotient of~$G$ by this subgroup.
We call $A$, $E$, and $G$, the first, second, and third
terms of the extension, respectively.
We always assume that the extension is \emph{abelian},
in other words, that $A$ is an abelian group.
In fact, in our application all extensions
are \emph{central}, that is, having $A$ as a central
subgroup of~$G$.
Since $A$ is abelian, the conjugation action
of~$G$ on~$E$ descends to an action of~$G$ on~$A$:
that is, $A$ is a $G$-module.
Conversely, starting from a $G$-module~$A$,
an extension of $G$ by~$A$
is an extension of $G$ by the abelian group~$A$
with the property that the $G$-action on~$A$
induced by the extension agrees with the given action.

A \emph{morphism} of extensions is a commutative diagram
\begin{center}
\begin{tikzcd}
1 \rar & A \dar\rar & E \dar\rar & G \dar\rar & 1 \\
1 \rar & A' \rar & E' \rar & G' \rar & 1.
\end{tikzcd}
\end{center}
As for extensions, we can speak of the first, second,
and third terms of a morphism,
namely, the maps $A\to A'$, $E\to E'$, and $G\to G'$,
respectively.
The isomorphisms in this category 
are precisely the morphisms whose terms
are all isomorphisms.
An easy diagram chase shows that a morphism
is an isomorphism as soon as its first
and third terms are isomorphisms.
For this reason, we often restrict attention
to the isomorphisms whose first and third terms
are equalities;
we call such an isomorphism an \emph{equivalence}.
Besides the isomorphisms,
there are two kinds of morphisms of interest.

First, given two $G$-modules $A$ and~$A'$
and a $G$-equivariant
homomorphism $\alpha\colon A\to A'$,
we can construct the morphism of extensions
\begin{center}
\begin{tikzcd}
1 \rar & A \dar{\alpha}\rar & E \dar\rar & G \dar[equals]\rar & 1 \\
1 \rar & A' \rar & \alpha_*(E) \rar & G \rar & 1
\end{tikzcd}
\end{center}
in which $\alpha_*(E)$ is the cokernel
of the map $A\to A'\rtimes E$
sending $a$ to $\alpha(a)a$.
We call the bottom extension (or sometimes, its second term)
the \emph{pushout} of the top extension.%
\footnote{Our pushout is usually not isomorphic to
the pushout in the category of groups.
The category-theoretic pushout is the amalgamated free product.}
The pushout morphism is universal
in the sense that any morphism with first term $\alpha$ 
from the top extension to an extension
of $G$ by~$A'$ factors through the pushout morphism.

Second, given a homomorphism $\gamma\colon G'\to G$,
we can construct the morphism of extensions
\begin{center}
\begin{tikzcd}
1 \rar & A \dar[equals]\rar & \gamma^*(E) \dar\rar &
	G' \dar{\gamma}\rar & 1 \\
1 \rar & A \rar & E \rar & G \rar & 1
\end{tikzcd}
\end{center}
in which $\gamma^*(E)$ is the pullback (of sets,
hence of groups) of the right square.
We call the top extension (or sometimes, its second term)
the \emph{pullback} of the bottom extension.
The pullback morphism is universal in the sense that
any morphism with third term $\gamma$
from an extension of $A$ by~$G'$
to the bottom extension factors through
the pullback morphism.

The theory of group extensions is relevant to us
because of its connection to group cohomology,
which relates to root numbers, in turn,
by Deligne's theorem.

Given an extension
\begin{center}
\begin{tikzcd}
1 \rar & A \rar & E \rar & G \rar & 1,
\end{tikzcd}
\end{center}
let $s\colon G\to E$ be a set-theoretic section of $E\to G$.
The formula
\[
z_s(g,g')=s(g)s(g')s(gg')^{-1}
\]
defines a $2$-cocyle $z_s\in\tn Z^2(G,A)$.
The cohomology class of the resulting cocycle
is independent of the choice of section:
any two such sections differ by a function $G\to A$,
and the coboundary of this function exhibits
a cohomology between the cocycles.

Conversely, given a $2$-cocycle $z\in\tn Z^2(G,A)$,
define the group extension
\begin{center}
\begin{tikzcd}
1 \rar & A \rar & A\boxtimes_zG \rar & G \rar & 1
\end{tikzcd}
\end{center}
in which $A\boxtimes_zG=A\times G$ with multiplication
\[
(a,g)\cdot(a',g') = (a\cdot{}^ga\cdot z(g,g'),gg').
\]
The equivalence class of the resulting extension
depends only on the cohomology class of the cocycle:
any function $G\to A$ whose coboundary exhibits
a cohomology between two different cocycles
gives rise to an equivalence of the corresponding extensions.

In summary, given a group $G$ and a $G$-module~$A$,
we have constructed a canonical bijection
between equivalence classes
of extensions of~$G$ by~$A$ 
and the cohomology group~$\tn H^2(G,A)$.
The bijection is compatible with pushforward
and pullback of cohomology classes and extensions.

The dictionary between extensions and cohomology classes
nicely answers, or rather, reformulates,
a natural question in the theory of group extensions:
when does the diagram
\begin{center}
\begin{tikzcd}
1 \rar & A \dar{\alpha}\rar & E \rar & G \dar{\gamma}\rar & 1 \\
1 \rar & A' \rar & E' \rar & G' \rar & 1
\end{tikzcd}
\end{center}
extend to a morphism of extensions?
The answer to the question is a special case
of our key lemma, \Cref{thm4},
and is also used in the proof of the lemma.

To answer the question, use the universal
properties of pullback and pushforward
to extend the candidate morphism to the diagram
\begin{center}
\begin{tikzcd}
1 \rar & A \dar{\alpha}\rar & E \dar\rar & G \dar[equals]\rar & 1 \\
1 \rar & A' \dar[equals]\rar & \alpha_*(E) \rar & G \dar[equals]\rar & 1 \\
1 \rar & A' \dar[equals]\rar & \gamma^*(E') \dar\rar & G \dar{\gamma}\rar & 1 \\
1 \rar & A' \rar & E' \rar & G' \rar & 1
\end{tikzcd}
\end{center}
It follows that the original diagram can be extended
to a morphism if and only if the middle extensions
in the diagram are isomorphic.
In other words, letting $c\in\tn H^2(G,A)$
and $c'\in\tn H^2(G',A')$ denote the cohomology classes
classifying the top and bottom extensions,
the original diagram can be extended to a morphism
if and only if
\[
\alpha_*(c) = \gamma^*(c') 
\qquad\tn{in $\tn H^2(G,A')$.}
\]

\subsection{Borel cohomology} \label{sec:gc:bc}
A sequence of topological groups
\begin{center}
\begin{tikzcd}
1 \rar & A \rar & E \rar & G \rar & 1
\end{tikzcd}
\end{center}
is a \emph{topological extension} of~$G$ by~$A$
if $A\to E$ is a closed subgroup
and the induced map from the cokernel
of $A\to E$ to~$G$,
where the source has the quotient topology,
is an isomorphism.
The pullback and pushout constructions
from \Cref{sec:gc:ge} work just as well
in this setting, provided all maps
in question are continuous.

Classifying topological extensions
is more subtle than classifying 
extensions of discrete groups,
however, because not every continuous surjection
of topological spaces admits a continuous section.
If we were to carry out the work of \Cref{sec:gc:ge}
in the category of topological groups,
where all maps are required to be continuous,
we would find that the continuous cohomology
$\tn H^2(G,A)$ classifies extensions of $G$ by $A$
that are \emph{topologically split}, that is,
whose second term is a direct product $G\times A$
as a topological space.
The collection of topologically split extensions
is much too small for most purposes.
For instance, the universal-cover group extension of a Lie group
is never topologically split,
but we need a theory that can see the spin extension
of the special orthogonal group because of
its great relevance to the study of Stiefel-Whitney classes.

To capture the topological extensions that
are not topologically split one must enlarge
the continuous cohomology group.
The correct enlargement is known as
\emph{Borel cohomology},%
\footnote{Borel cohomology is named after
Émile, not Armand.
It is sometimes also called
\emph{Moore cohomology}.}
a variant of discrete group cohomology
in which one requires that cochains
be Borel measurable.
I refer the reader
to Stasheff's survey article
\cite{stasheff78}
as well as Moore's papers on the subject
\cite{moore64a,moore64b,moore76a,moore76b}
for more information on Borel cohomology.

\begin{theorem} \label{thm22}
Let $G$ and $A$ be separable 
locally compact groups with $A$ abelian.
There is a canonical natural isomorphism
between $\tn H_\tn{Borel}^2(G,A)$
and the set of equivalence classes
of topological extensions of $G$ by~$A$.
\end{theorem}

\begin{proof}
Mackey proved this result in his thesis \cite{mackey57}.
The same construction as in \Cref{sec:gc:ge}
yields the bijection,
with the following additional argument.
Given an extension of $G$ by~$A$,
one must find a section 
whose associated cocycle is Borel measurable.
This is Mackey's Théorème~3, and
its proof shows that the section
need only be Borel-measurable.
Conversely, given a Borel-measurable
cocycle $z:G\times G\to A$,
one must show endow the extension
$A\boxtimes_z G$ with a locally compact topology
making it an extension of topological groups.
This is Mackey's Théorème~2.
A uniqueness statement in that theorem
ensures that the topology on the extension
can be recovered from any measurable
cocycle classifying it.
\end{proof}

\begin{example} \label{thm31}
Using the fact that $\C^\times$ is connected
and $\{\pm1\}$ is discrete,
it is easy to show that the continuous cohomology group
$\tn H^2(\C^\times,\{\pm1\})$ is trivial.
On the other hand, since $\C^\times=W_\C$,
we know \cite[(4.2)]{deligne76}
that $\tn H^2_\tn{Borel}(\C^\times,\{\pm1\})\simeq\{\pm1\}$.
The nontrivial cohomology class here represents
the topologically non-split squaring extension
\begin{center}
\begin{tikzcd}
1 \rar &
\{\pm1\} \rar &
\C^\times \rar{z\mapsto z^2} &
\C^\times \rar &
1.
\end{tikzcd}
\end{center}
\end{example}

Since Borel cohomology is not widely used
in the Langlands program,
we point out several relevant properties.

First, in the nonarchimedean case,
the pullback of a Borel cocycle
along an $L$-parameter
is automatically continuous.

\begin{lemma}
Let $W$ and~$G$ be topological groups,
let $A$ be a continuous $G$-module,
and let $\varphi:W\to G$ be a continuous homomorphism.
If $\varphi$ factors through a discrete quotient of~$W$
then pullback by~$\varphi$ induces a map
\[
\varphi^*:\tn H_\tn{Borel}^i(G,A) \to \tn H^i(W,A).
\]
\end{lemma}

\begin{proof}
Pullback along~$\varphi$ takes
an arbitrary cocycle $G^n\to A$
to a cocycle $W^n\to A$ that is constant on cosets
of an open subgroups.
So the pulled-back cocycle is continuous.
\end{proof}

Second, in the archimedean case,
Borel $\tn H^1$ works just as well
as continuous $\tn H^1$
in the local Langlands correspondence.

\begin{lemma} \label{thm32}
Let $G$ and $A$ be separable completely metrizable
topological groups, in other words, Polish groups,
and let $A$ carry a continuous $G$-module structure.
The natural map
\[
\tn H^1(G,A) \to \tn H^1_\tn{Borel}(G,A)
\]
is an isomorphism.
\end{lemma}

Results of this kind are known as \emph{automatic continuity}.

\begin{proof}
A theorem of Banach and Pettis \cite{pettis50},
nicely explained in Rosendal's overview
of automatic continuity \cite[Theorem~2.2]{rosendal09},
proves the result when the action of $G$ on~$A$ is trivial,
in which case $\tn H^1(G,A) = \Hom_\tn{cts}(G,A)$.
In general, use their theorem together with
the fact that a crossed homomorphism
$G\to A$ is the same as a homomorphism
$A\rtimes G\to G$ that restricts
to the identity on~$G$.
\end{proof}

\begin{remark}
It would be interesting to see
if Clausen and Scholze's
condensed mathematics \cite{clausen_scholze19}
offers a replacement to Borel cohomology.
We will not pursue this idea here.
\end{remark}

\subsection{The key lemma} \label{sec:gc:kl}
The following \namecref{thm4} is
the main technical tool underpinning
the proof of \Cref{thm1}.
There is nothing essential in its use of Borel cohomology.

\begin{lemmax} \label{thm4}
Let $A$, $A'$, $E$, $G$, $G'$, and~$W$
be topological groups with $A$ and~$A'$ abelian.
Assume that either all groups are discrete or
all groups are separable and locally compact.
Let the group $W$ act continuously on the groups~$G$ and~$E$
by group automorphisms,
let $E\to G$ be a continuous $W$\!-equivariant homomorphism,
and let
\begin{center}
\begin{tikzcd}
1 \rar &
A \rar\dar{\alpha} &
E \rar\dar{\varepsilon} &
G \rar\dar{\gamma} &
1 \\

1 \rar &
A' \rar &
E' \rar &
G' \rar &
1
\end{tikzcd}
\end{center}
be a morphism of topological extensions.
Let $f\colon W\to G'$ be a continuous homomorphism
such that the map $\gamma f\colon G\rtimes W\to G'$
is a homomorphism.
Consider the diagram
\begin{center}
\begin{tikzcd}
1 \rar &
A \rar\dar{\alpha} &
E\rtimes W \rar &
G\rtimes W \rar\dar{\gamma f} &
1  &
(c) \\

1 \rar &
A' \rar &
E' \rar &
G' \rar &
1 &
(c').
\end{tikzcd}
\end{center}
with top and bottom extensions classified
by the cohomology classes $c$ and~$c'$,
respectively,
and let $p\colon G\rtimes W\to W$ denote
the canonical projection.
Then
\[
(\gamma f)^*(c') = \alpha_*(c)\cdot p^* f^*(c').
\]
\end{lemmax}

The assumption on the topologies
of the groups involved implies,
by \Cref{thm22} in the locally compact case
and the discussion of \Cref{sec:gc:ge}
in the discrete case,
that the relevant extensions
are classified by Borel cohomology classes.

\begin{proof}
The proof rests on an understanding
of $\tn H_\tn{Borel}^2(G\rtimes W,A)$.
The semidirect product fits into
a split short exact sequence
\begin{center}
\begin{tikzcd}
1 \rar &
G \rar{i} &
G \rtimes W \rar{p} &
W \rar \lar[bend left,swap,dashed]{s} &
1
\end{tikzcd}
\end{center}
in which $i\colon G\to G\rtimes W$ is the canonical inclusion,
$p\colon G\rtimes W\to W$ is the canonical projection,
and $s\colon W\to G\rtimes W$ is the canonical inclusion.
This sequence dualizes to a split exact sequence
of abelian groups
\begin{center}
\begin{tikzcd}
1 \rar &
\tn H_\tn{Borel}^2(W,A) \rar{p^*} &
\tn H_\tn{Borel}^2(G\rtimes W,A) \rar{i^*}
\lar[bend left,swap,dashed]{s^*} &
\tn H_\tn{Borel}^2(G,A).
\end{tikzcd}
\end{center}
The map $p^*$ realizes $\tn H_\tn{Borel}^2(W,A)$
as a direct summand of $\tn H_\tn{Borel}^2(G\rtimes W,A)$
with a canonical complement, the kernel of~$s^*$.
The map $p^*s^*$ is projection onto the summand,
and the map $i^*$ identifies its complement,
the kernel of~$s^*$, with a subgroup of $\tn H_\tn{Borel}^2(G,A)$.
To prove the lemma, it therefore suffices to show that
$s^*(\gamma f)^*(c') =  f^*(c')$ and that
$i^*(\gamma f)^*(c') = i^*\alpha_*(c)$.
The first equation follows
from the identity $ f = (\gamma f)\circ s$.
The second equation amounts to showing that
$\gamma^*(c') = \alpha_*i^*(c)$,
and this is a consequence of the existence of~$\varepsilon$.
\end{proof}

Later, in \Cref{sec:cc:inv},
we need the following compatibility condition.

\begin{lemma} \label{thm16}
In the setting of \Cref{thm4},
a Borel-measurable crossed homomorphism $\varphi\colon W\to G$
maps under the coboundary $\tn H_\tn{Borel}^1(W,G)
\to\tn H_\tn{Borel}^2(W,A)$
to the pullback of~$c$ along the homomorphism
$\varphi\cdot\tn{id} \colon W\to G\rtimes W$.
\end{lemma}

By \Cref{thm32}, we could have written
``continuous'' here instead of ``Borel-measurable''.

\begin{proof}
Let $s\colon G\to E$ be a Borel-measurable
section of $E\to G$
and let $\widetilde\varphi=s\circ\varphi$
be a lift of $\varphi$ to~$E$.
The coboundary of $\varphi$
is given by the formula \cite[Chapter~I.5.6]{serre02}
\[
(w,w') \mapsto \widetilde\varphi(w)
\cdot{}^w\widetilde\varphi(w')
\cdot\widetilde\varphi(ww')^{-1}.
\]
On the other hand, the class~$c$ is represented
by the $2$-cocycle
\[
(g,w;g',w') \mapsto s(g)w\cdot s(g')w'
\cdot(s(g\cdot{}^wg')ww')^{-1}
= s(g)\cdot{}^ws(g')\cdot s(g\cdot{}^wg')^{-1}
\]
corresponding to the section
$(s,\tn{id})\colon G\rtimes W\to E\rtimes W$.
Pulling back this function along~$\varphi$
replaces $g$ by $\varphi(w)$ and recovers
the coboundary of~$\varphi$.
\end{proof}

For our application of \Cref{thm4} to the proof of \Cref{thm1},
the group~$W$ is the Weil group,
the top extension is the universal cover of the dual group,
the bottom extension is the universal cover
of a complex orthogonal group,
and the morphism between them arises
from the given orthogonal representation
$\rep\colon {}^LG\to\tn O(V)$:
\begin{equation} \label{thm11}
\begin{tikzcd}
1 \rar &
\pi_1(\widehat G) \rar\dar{e_\rep} &
\widehat G_\tn{univ} \rar\dar &
\widehat G \rar\dar{\rep|_{\widehat G}} &
1 \\
1 \rar &
\{\pm1\} \rar &
\Pin(V) \rar &
\tn O(V) \rar &
1 &
(c_\tn{pin}).
\end{tikzcd}
\end{equation}
Here $e_\rep$ is the ``spin character'',
which we study in \Cref{sec:spin},
and the class $c_\tn{pin}$
classifies the bottom extension.
Let $c_G\in\tn H^2_\tn{Borel}({}^LG,\pi_1(\widehat G))$
classify the extension $\widehat G_\tn{univ}\rtimes W_k$
of ${}^LG$ by~$\pi_1(\widehat G)$,
as in the \namecref{thm4}.

With this setup,
we prove the theorem by pulling back the conclusion
of \Cref{thm4} along the given $L$-parameter.
After pullback, the quotient of the cohomology classes
$\rep^*(c_\tn{pin})$ and $p^*\rep|_{W_k}^*(c_\tn{pin})$
becomes a quotient of root numbers
and the cohomology class $e_{\rep,*}(c_G)$
becomes the value of a central character on an involution.
The goal of the remainder of the article
is to explain these identifications,
thereby proving \Cref{thm1}.

\section{Stiefel-Whitney classes} \label{sec:sw}
Deligne's formula for orthogonal root numbers
is a key technical tool supporting
the main results of this article.
To use his formula effectively,
we need a workable definition
of the second Stiefel-Whitney class.
The goal of this largely expository section
is to explain how to interpret
in terms of group cohomology
the second Stiefel-Whitney class
of a bounded complex orthogonal representation
of a countable discrete group.
In the end, the class is
the pullback of a certain pin-group extension.

This interpretation is surely well known
to the experts, and special cases
have already appeared in the literature,
for instance, in a paper of
Gunarwardena, Kahn, and Thomas
on real orthogonal representations of finite groups
\cite{gunarwardena_kahn_thomas89}.
We generalize their results
by working with countable discrete groups
instead of finite groups
and complex representations
instead of real representations.

\subsection{Classifying spaces}
One way to construct characteristic classes
is to pull them back from universal
cohomology classes of a certain classifying space.
In this subsection, we review the theory
of classifying spaces,
loosely following Mitchell's notes on
classifying spaces \cite{mitchell11}
and Section~6 of Stasheff's
survey article \cite{stasheff78}.

Let $G$ be a Lie group.
We assume $G$ to be second countable
but we do not assume $G$ to be connected.
So $G$ could be a complex reductive group
or a countable discrete group.

Let $P\to B$ be a principal $G$-bundle.
Although $G$ is a manifold
we do not impose any smoothness
or differentiability assumption on bundles:
they are simply continuous.
If $P$ is weakly contractible,
that is, having trivial homotopy groups,
then we call $B$ a \emph{classifying space} for~$G$
and $P$ a \emph{universal $G$-bundle}.
The bundle is universal in the following sense:
for every CW-complex~$X$, the canonical map 
from the set of homotopy classes of maps $X\to B$
to the set of equivalence classes
of principal $G$-bundles over~$X$,
defined by pulling back
the universal $G$-bundle, is a bijection.
As the universal property makes reference
only to homotopy classes of maps,
a classifying space for~$G$ is defined uniquely
only up to homotopy equivalence.

It turns out, though this is not clear from the definition,
that classifying spaces exist for every~$G$.
Often we can construct the classifying space by hand.
For example, the classifying space
of a compact orthogonal group of rank~$n$
is the Grassmannian of $n$-planes in~$\R^{\oplus\N}$.
But for a general group such a geometric construction is difficult,
and we can instead construct the classifying space
by simplicial methods
\cite[Chapter 16, Section 5]{may99}. %% concise topology book
The simplicial construction of classifying spaces
makes clear their functoriality:
a homomorphism $\rep\colon G\to H$
of topological groups gives rise
to a map $B\rep\colon BG\to BH$.
Functoriality also follows from Yoneda's Lemma,
without needing to choose a specific model
for the classifying space:
the map $B\rep$
represents the balanced product functor
$P\mapsto H\times_G P$
from principal $G$-bundles
to principal $H$-bundles.

The homotopy class of the classifying space~$BG$
depends only on homotopy type of the group~$G$
in the following sense:
any group homomorphism $G\to H$ that is a homotopy equivalence
induces a homotopy equivalence of classifying spaces.
At the same time, a theorem of Iwasawa
and Malcev \cite[Section~7]{samelson52}
implies that the inclusion into~$G$
of a maximal compact subgroup~$K$ is a homotopy equivalence.
Hence the map $BK\to BG$ is a homotopy equivalence.

Universal characteristic classes live
in the singular cohomology of classifying spaces.
To translate characteristic classes into the language
of group cohomology, therefore,
we should strive to interpret the singular cohomology
of a classifying space in terms of group cohomology.
In his thesis \cite{wigner73},
Wigner gave such an interpretation
using Borel cohomology.

\begin{theorem} \label{thm3}
Let $A$ be a discrete abelian group.
There is a canonical natural isomorphism
\[
\tn H_\tn{sing}^i(B(\cdot),A)
\simeq H_\tn{Borel}^i(\cdot,A)
\]
of contravariant functors
from the category of Lie groups
to the category of abelian groups.
\end{theorem}

\begin{proof}
Let $G$ be a Lie group.
Wigner defined cohomology groups
$\tn H_\tn{Wig}^\bullet(G,A)$
which he showed to agree with
the Borel cohomology groups
$\tn H_\tn{Borel}^\bullet(G,A)$.
Wigner's Theorem~4 shows that
$\tn H^\bullet(G,A)$ agrees with the 
cohomology of the constant sheaf
$\ul A$ on the classifying space~$BG$.
Since $G$ is a Lie group,
$BG$ is sufficiently nice
(homotopy equivalent to a CW complex, say)
that this sheaf cohomology
agrees with the singular cohomology
$\tn H_\tn{sing}^i(BG,A)$ \cite{sella16}.
\end{proof}

\subsection{Definition for vector bundles}
Let $X$ be a CW complex.
The theory of Stiefel-Whitney classes
assigns to a rank-$n$ real vector bundle~$V$
over~$X$ a family of cohomology classes
\[
w_i(V)\in\tn H^i_\tn{sing}(X,\F_2),
\qquad i = 0,1,\dots,n.
\]
These characteristic classes,
along with others like the Chern classes,
provide a powerful algebraic framework
for computations with vector bundles.
The standard source for the subject
is Milnor and Stasheff's book
on characteristic classes \cite{milnor_stasheff74};
for our application, Chapters~4 to~9 and~14
are especially relevant.

Stiefel-Whitney classes
are characterized abstractly,
via a series of axioms:
unitality, naturality,
multiplicativity of the total class,
and nontriviality.
To show that these axioms do indeed define
a collection of cohomology classes,
and that this collection is unique,
one defines the classes
as the pullback of certain universal cohomology classes
on a classifying space.
The naturality axiom forces such a description.

The classifying space $B\!\GL_n(\R)$
is the Grassmannian of $n$-planes in~$\R^{\oplus\N}$,
topologized as a direct limit, and
the universal bundle over~$B\!\GL_n(\R)$
is just the tautological bundle over the Grassmannian,
whose fiber over a point is the $n$-plane
the point represents.
The $\F_2$ singular cohomology ring of
this infinite Grassmannian is a graded polynomial ring
with one generator for each $i=1,2,\dots,n$.
We call the generator in degree~$i$ the $i$th
\emph{universal singular Stiefel-Whitney class}.
Given a rank-$n$ real vector bundle~$V$
on~$X$ classified by the map $f_V\colon X\to B\!\GL_n(\R)$,
we define the $i$th Stiefel-Whitney class of~$V$
as the image under the pullback
\[
f_V^*: \tn H^i_\tn{sing}(B\!\GL_n(\R),\F_2)
\to \tn H^i_\tn{sing}(X,\F_2)
\]
of the $i$th universal singular Stiefel-Whitney class.

Since the compact orthogonal group~$\tn O_n$
of rank~$n$ is a maximal compact subgroup of~$\GL_n(\R)$,
the classifying spaces of the two groups
are homotopy equivalent to each other
and real vector bundles are classified by maps to~$B\tn O_n$.
In the literature one often works 
with the classifying spaces of this maximal compact subgroup
instead of the ambient general linear group,
as we do here.

\subsection{Definition for representations}
It is now clear how to define the Stiefel-Whitney classes
of a real representation.
Let $w_{i,\tn{univ}}\in
\tn H^i_\tn{Borel}(\GL_n(\R),\{\pm1\})$
denote the Borel cohomology class corresponding,
via \Cref{thm3}, to the $i$th universal
singular Stiefel-Whitney class.
We call $w_{i,\tn{univ}}$
the \emph{$i$th universal Stiefel-Whitney class}.
Set $w_{i,\tn{univ}} \defeq 0$ if $i<0$ or $i>n$.

\begin{definition}
Let $G$ be a Lie group.
The \emph{$i$th Stiefel-Whitney class}
of a real representation
$\rep:G\to\GL_n(\R)$ is the cohomology class
\[
w_i(\rep) \defeq \rep^*(w_{i,\tn{univ}})
\in \tn H^i_\tn{Borel}(G,\{\pm1\}).
\]
\end{definition}

This definition is not enough for us, however.
Since the relevant representations of $L$-groups are complex,
not real, we need to define the Stiefel-Whitney classes
of a complex representation, the goal of this subsection.
In outline, to make the definition we simply require
that base change from $\R$ to~$\C$
leave Stiefel-Whitney classes unchanged.
This stipulation defines Stiefel-Whitney classes
for all complex representations that descend to~$\R$,
in particular, the bounded orthogonal representations.

An \emph{$\R$-structure} on a complex vector space~$V$
is a real subspace $V_0\subseteq V$
such that the map $V_0\otimes_\R\C\to V$
is an isomorphism.
An \emph{$\R$-structure} on a complex representation
$(\rep,V)$ is a real representation $(\rep_0,V_0)$
such that $V_0$ is an $\R$-structure on~$V$
and $\rep$ factors through~$\rep_0$.
An \emph{isomorphism} of $\R$-structures
on a complex representation is an isomorphism
of the corresponding real representations.

It is not the case that every complex orthogonal
representation admits an $\R$-structure:
for instance, the two-dimensional representation of~$\Z$
that sends~$1$ to the diagonal matrix
with entries $(2i,-i/2)$
is orthogonal but does not admit an $\R$-structure
because its character takes imaginary values.
However, if the representation is
in addition \emph{bounded},
meaning the closure of its image is compact,
then it does admit an $\R$-structure.

\begin{lemma}\label{thm5}
Every bounded complex orthogonal representation
admits an $\R$-structure.
\end{lemma}

\begin{proof}
An exercise in point-set topology shows
that the closure of the image of
the representation is again a group,
and by hypothesis the closure is compact.
Hence the representation factors through
some compact orthogonal group~$\tn O_n$,
where $n$ is the complex dimension
of the representation, because $\tn O_n$
is a maximal compact subgroup of
the complex orthogonal group.
\end{proof}

The base change problem appears in many
other settings than representation theory.
For example, we can study the base change
of varieties from $\R$ to~$\C$.
In that setting, it can happen that $\R$-varieties
are nonisomorphic but become isomorphic
upon base change to~$\C$.
In our setting the story is simpler:
any two real forms of a (semisimple)
representation must be isomorphic.

\begin{lemma} \label{thm35}
Let $G$ be a group and
$\rep:G\to\GL(V)$ a finite-dimensional complex representation.
If $\rep$ is semisimple then any two
of its $\R$-structures are isomorphic.
\end{lemma}

\begin{proof}
The isomorphism classes of
$\R$-structures on $\rep$ are classified
by the Galois cohomology set
$\tn H^1(\Gamma\!_\R,\Aut_G(\rep))$.
Since $\rep$ is semisimple,
$\Aut_G(\rep)$ is a product
of complex general linear groups.
Hence the cohomology set is trivial
by Hilbert's Theorem~90.
\end{proof}

By the \namecref{thm35},
the following definition
does not depend on the choice of $\R$-structure.

\begin{definition}
Let $G$ be a Lie group and
let $(\rep,V)$ be a semisimple complex representation of~$G$
that admits an $\R$-structure $(\rep_0,V_0)$.
The \emph{$i$th Stiefel-Whitney class} of~$(\rep,V)$ is
\[
w_i(\rep) \defeq w_i(\rep_0)
\in \tn H^i_\tn{Borel}(G,\{\pm1\}).
\]
\end{definition}

We are most interested in the second Stiefel-Whitney class
because of its appearance in Deligne's theorem on root numbers.

\begin{lemma} \label{thm36}
Let $V_0$ be a real anisotropic quadratic space
and let $V \defeq V_0\otimes_\R\C$,
a complex quadratic space.
There exists an extension~$\Pin(V)$
of $\tn O(V)$ by~$\{\pm1\}$,
called the (complex) \emph{pin group},
with the following property: the class
$c_\tn{pin}\in\tn H^2_\tn{Borel}(\tn O(V),\{\pm1\})$
that classifies $\Pin(V)$ is the pullback of
the second universal Stiefel-Whitney class
\[
w_{2,\tn{univ}}\in
\tn H^2_\tn{Borel}(\tn O(V_0),\{\pm1\})
\simeq \tn H^2_\tn{Borel}(\GL(V_0),\{\pm1\}).
\]
\end{lemma}

We call the class
$c_\tn{pin}\in\tn H^2_\tn{Borel}(\tn O(V),\{\pm1\})$
the \emph{pin class}.

\begin{proof}
Let $\Pin(V_0)$ be the extension of $\tn O(V_0)$
by~$\{\pm1\}$ classified by~$w_{2,\tn{univ}}$.
It is well-known that $\Pin(V_0)$
is (the rational points of)
the standard (real) algebraic pin group
defined using the Clifford algebra $\Cl(V_0)$;
see, for instance, Appendix~I of \cite{frohlich85}
or the introduction to \cite{gunarwardena_kahn_thomas89}.

To be specific, the group $\GPin(V_0)$ is
the stabilizer of the subspace $V_0\subseteq\Cl(V_0)$
under the conjugation action of
the units group $\Cl(V_0)^\times$ on~$\Cl(V_0)$.
The main anti-involution of~$\Cl(V_0)$
is the automorphism~$\alpha$ induced
by the order-reversing automorphism
$v_1\otimes\cdots\otimes v_d
\mapsto v_d\otimes\cdots\otimes v_1$
of the tensor algebra on~$V_0$.
The spinor norm is the homomorphism $\Cl(V_0)\to\C^\times$
sending $x\in\Cl(V_0)$ to $x\alpha(x)$.
Finally, $\Pin(V_0)$ is the kernel of the spinor norm.

It is now clear that we may
take as $\Pin(V)$ the complex algebraic group
obtained from (the algebraic group underlying)
$\Pin(V_0)$ by base change from $\R$ to~$\C$.
\end{proof}

It follows from the \namecref{thm36} that
the second Stiefel-Whitney class
of a bounded complex orthogonal representation
$r:G\to\tn O(V)$ is classified
by the pullback of group extensions $r^*\Pin(V)$.
This conclusion is our final reformulation
of the second Stiefel-Whitney class
in the language of group cohomology.

\begin{remark}
There are two variant and non-isomorphic
definitions of the pin group,
stemming from the fact that
the elements $w_{2,\tn{univ}}$ and $w_{1,\tn{univ}}^2$
of $\tn H^2_\tn{Borel}(\tn O(V),\{\pm1\})$ are distinct.
Conrad's SGA~3 article \cite{conrad14},
whose notation for the Pin group agrees with ours,
nicely explains the difference;
Remarks C.4.9 and C.5.1 are especially relevant.
The other variant of the pin group,
which we do not use here,
is due to Atiyah, Bott, and Shapiro
\cite{atiyah_bott_shapiro63},
and is denoted by $\Pin^-$ in Conrad's article.
The definition is similar to ours but one
modifies the spinor norm by a sign twist.
\end{remark}

\subsection{Deligne's theorem}
We start by defining the Stiefel-Whitney classes
of a complex orthogonal representation
of the Weil group.
When the field~$k$ is archimedean,
this task is already complete because
$W_\R$ and~$W_\C$ are Lie groups.
When $k$ is nonarchimedean,
we use the fact that every complex representation
of~$W_k$ factors through a discrete quotient.

\begin{definition} \label{thm33}
Let $k$ be nonarchimedean,
let $V$ be a complex quadratic space,
and let $\rep\colon W_k\to\tn O(V)$ be a
complex orthogonal representation of~$W_k$.
The $i$th \emph{Stiefel-Whitney class} of $(\rep,V)$
is the image of $w_i(\rep)$ under the inflation map
\[
\tn H^i(W_k/\ker\rep,\{\pm1\})
\to \tn H^i(W_k,\{\pm1\}).
\]
\end{definition}

The evident compatibility of Stiefel-Whitney classes
with inflation shows that we are free to replace $\ker\rep$
by any open subgroup of~$W_k$ on which $\rep$ is trivial.
We can now state our reformulation of Deligne's theorem
on root numbers.

\begin{theorem} \label{thm19}
Let $\rep\colon W_k\to\tn O(V)$ be a bounded
complex orthogonal representation
and let $c_\tn{pin}\in\tn H_\tn{Borel}^2(\tn O(V),\{\pm1\})$
be the pin class of \Cref{thm36}.
Then
\[
\frac{\omega(\rep)}{\omega(\det\rep)}
= \sgn\rep^*(c_\tn{pin}).
\]
\end{theorem}

The statement of the \namecref{thm19} uses that
the group $\tn H^2(W_k,\{\pm1\})$ is cyclic of order two.
The function $\sgn$ was defined in \Cref{sec:intro:not};
it uniquely identifies
this group with the group~$\{\pm1\}$.

\begin{proof}
Applying Deligne's theorem
\cite[Proposition~5.2]{deligne76}
to the virtual representation
\[
\rep - \det\rep - (\dim\rep-1)\cdot\tn{triv}
\]
of dimension zero and determinant
shows that the lefthand side
of the equation equals $\sgn(w_2(\rep))$.
And by \Cref{thm36},
$w_2(r) = r^*(c_\tn{pin})$.
\end{proof}

\section{Spin lifting} \label{sec:spin}
Let $G$ be a complex reductive group,
identified with its set of $\C$-points.
In our application~$G$ will be
the Langlands dual of a reductive $k$-group,
but in this section we restrict attention
to complex groups so there is no need
to decorate~$G$ with a hat.

Consider an (algebraic, complex)
orthogonal representation $\rep\colon G\to\tn O(V)$.
As $G$ is (by assumption) connected,
the representation~$\rep$ factors
through the identity component $\SO(V)$ of~$\tn O(V)$,
and we may without loss of generality
replace $\tn O(V)$ by~$\SO(V)$ as the target of~$r$.

The special orthogonal group~$\SO(V)$ is not simply connected.
Supposing that $\dim V>2$,
its universal cover $p\colon \Spin(V)\to\SO(V)$, a double cover,
is an algebraic group called the \emph{spin group}.
We can construct the spin group as a subgroup of
the units group of a Clifford algebra,
though for our purposes, we can understand the group
using the combinatorics of root systems alone.
When $\dim V=2$, so that $\SO(V)=\G_{\tn m}$
is a one-dimensional torus, we define
$\Spin(V)=\G_{\tn m}$ with double cover
$p\colon \Spin(V)\to\SO(V)$ the squaring map.

The existence of the spin group creates
a dichotomy in the orthogonal representations
$\rep\colon G\to\SO(V)$ of~$G$:
either the representation lifts to the spin group
or it does not.
We can study and refine this lifting question
by passing to the universal cover~$G_\tn{univ}$ of~$G$:
\[
G_\tn{univ} = \frak z\times G_\tn{sc}
\]
where $\frak z$ is the Lie algebra of the center~$Z$ of~$G$
and $G_\tn{sc}$ is the simply-connected cover
of the derived subgroup of~$G$.
Standard facts about covering spaces imply
that $\rep$ lifts to a homomorphism
$\rep_\tn{univ}\colon G_\tn{univ}\to\Spin(V)$.
Restricting $\rep_\tn{univ}$
to the kernels of the projections
yields the following commutative diagram,
which is essentially equivalent to \eqref{thm11}:
\[
\begin{tikzcd}
1 \rar &
\pi_1(G) \rar\dar{e_\rep} &
G_\tn{univ} \rar\dar{\rep_\tn{univ}} &
G \rar\dar{\rep} &
1 \\
1 \rar &
\{\pm1\} \rar &
\Spin(V) \rar &
\SO(V) \rar &
1.
\end{tikzcd}
\]
Then $\rep$ lifts to the spin group
if and only if the character
$e_\rep\colon \pi_1(G)\to\{\pm1\}$,
which we call the \emph{spin character} of~$\rep$,
is trivial.

Our goal in this section is to give
a formula for the spin character of a representation
in terms of its weights.
We start in \Cref{sec:spin:crit}
with a criterion for~$\rep$
to lift to the spin group.
From this criterion we then
deduce, in \Cref{sec:spin:char},
a formula for the spin character.

All that is new here is our exposition:
this calculation forms part of the canon
of the representation theory of compact Lie groups.
With that said, there is a small discrepancy
between our answer and Bourbaki's,
which I believe to be an error on Bourbaki's part.
This discrepancy is discussed in \Cref{sec:spin:bourbaki}.

Given a representation $(\rep,V)$ of~$G$
and a maximal torus~$T$,
let $\Phi(V)\subseteq X^*(T)$
denote the set of weights of~$T$ on~$V$.

\subsection{Lifting criterion} \label{sec:spin:crit}
The goal of this subsection is to describe
which orthogonal representations of~$G$
lift to the spin group.
We start with the essential case where $G=T$ is a torus;
the general case reduces easily to this one.

Let $S$ be a maximal torus of~$\SO(V)$
and let $\widetilde S\subseteq\Spin(V)$
be its preimage in the spin group, again a maximal torus.
Any homomorphism $T\to\SO(V)$ can be conjugated to take values in~$S$,
after which point the lift to the spin group,
if it exists, factors through~$\widetilde S$.
Passing to character lattices,
the lifting question
reduces to the algebra question
of whether the homomorphism
$X^*(S)\to X^*(T)$ can be extended
to the group~$X^*(\widetilde S)$,
which contains~$X^*(S)$ as an index-two
subgroup because $\widetilde S\to S$ is a double-cover.

\begin{center}
\begin{tikzcd}
&
\widetilde S \dar{p} \\
T \rar{\rep} \urar[dashed]{\widetilde\rep} &
S
\end{tikzcd}
\qquad\qquad
\begin{tikzcd}
&
X^*(\widetilde S) \dlar[swap,dashed]{\widetilde\rep^*} \\
X^*(T) &
X^*(S) \uar[swap]{p^*} \lar[swap]{\rep^*}
\end{tikzcd}
\end{center}

To solve this extension problem,
we need to understand the double cover $\widetilde S\to S$
as well as the relationship between
the homomorphism $T\to S$ and the weights
of the original orthogonal representation.
Let $n$ be the rank of~$\Spin(V)$,
so that $\dim V=2n$ or $2n+1$.
Fix a set~$I$ of cardinality~$\dim V$
equipped with an involution~$i\mapsto -i$
that fixes no element of~$I$ when $\dim V$ is even
and exactly one element of~$I$, which we denote by~$0$,
when $\dim V$ is odd.

First, the double cover.
Choose an $I$-indexed basis $e=(e_i)_{i\in I}$ of~$V$
for which $\langle e_i,e_j\rangle = [i = -j]$;%
\footnote{Here $[\cal P]$ is the Iverson bracket
popularized by Knuth \cite{knuth92}:
it equals $0$ if the property~$\cal P$ is false
and $1$ if $\cal P$ is true.}
we call such a basis a \emph{Witt basis}.
Let $S_e$ denote the group of
$s=(s_i)_{i\in I}\in\G_{\tn m}^I$ such that
$s_i\cdot s_{-i} = 1$ and such that,
when $\dim V$ is odd, $s_0=1$.
The rank-$n$ torus~$S_e$ acts on~$V$ by
\[
s\cdot e_i \defeq s_ie_i,
\]
realizing $S_e$ as a subgroup of~$\SO(V)$.
Evidently $S_e$ is a maximal torus,
and every other maximal torus of~$\SO(V)$ arises
from a Witt basis~$e$ by this construction.
Passing to the character lattice,
describing a basis of~$X^*(S_e)$ requires a choice
of gauge $p\colon I\setminus\{0\}\to\{\pm1\}$,
that is, a negation-equivariant function.
Given~$p$, say $i>0$ if $p(i)=+1$
and $i<0$ if $p(i)=-1$,
for $i\in I$.
The characters $f_i\colon s\mapsto s_i$
for $i>0$ give a basis for~$X^*(S)$.

We can use the Clifford-algebra description of $\Spin(V)$,
or even easier, the Bourbaki root-system tables
\cite[Planches]{bourbaki_lie4-6},
to work out the character lattice of~$X^*(\widetilde S)$.
Taking $f_p=(f_i)_{i>0}$ as a basis for $X^*(S)_\Q$, 
the character lattice of~$X^*(\widetilde S)$ is the set
of elements of $\tfrac12 X^*(S)$ whose coordinates
are all integers or all half-integers.
In particular, $X^*(\widetilde S)$ is generated
by $X^*(S)$ and the vector
$\tilde f_p = \tfrac12\sum_{i>0} f_i$.
It follows that for $A$ an abelian group,
a homomorphism $X^*(S)\to A$ extends to a homomorphism
$X^*(\widetilde S)\to A$ if and only if,
letting $a_i$ denote the image of~$f_i$,
the sum $\sum_{i>0} a_i$ lies in~$2A$.
If this property is satisfied then $\tilde f_p$
can map to any element whose double is $\sum_{i>0} a_i$.
When $A$ is $2$-torsionfree there is
at most one such element,
hence at most one extension.

Consequently, $\widetilde S$ can be described as
the quotient $S/B$ where $B$ is the group of
$(\varepsilon_i)_{i\in I}\in\{\pm1\}^I$
with $\varepsilon_i = \varepsilon_{-i}$,
$\varepsilon_0=1$,
and $\prod_{i>0}\varepsilon_i = 1$.
In this description the twofold cover
$\widetilde S\to S$ is induced
by the squaring map on~$S$.

Next, the weights.
Since $V$ is an orthogonal representation of~$T$,
its set of weights is negation-invariant.
Let $p\colon \Phi(V)\setminus\{0\}\to\{\pm1\}$ be a gauge.
Using the gauge we can write down an orthogonal decomposition
\[
V = V_0\oplus\bigoplus_{\alpha>0}(V_\alpha\oplus V_{-\alpha})
\]
in which $V_\alpha$ is the orthogonal complement
of $V_{-\alpha}$ when $\alpha\neq0$.
At this point the gauge is only a notational convenience
since the summands in the decomposition do not depend on it.
Choose a Witt basis $(e_i)_{i\in I}$ for~$V$
consisting of weight vectors of~$S$
and let $\alpha_i$ denote the weight of~$e_i$.
If $\dim V$ is odd then $e_0$ is of weight~$0$.
The map $X^*(S)\to X^*(T)$ dual to the homomorphism $T\to S$
sends the basis vector~$f_i$ to the character~$\alpha_i$.

These two analyses combine to a criterion
for the map $T\to S$ to lift to $\widetilde S$.
Choose a gauge~$p$ on~$X^*(T)$ and define the element
\begin{equation} \label{thm8}
\rho_\rep\defeq\frac12\sum_{\alpha>0}
(\dim V_\alpha)\alpha \in \frac12 X^*(T).
\end{equation}
Although $\rho_\rep$ depends on the choice of gauge,
we will only ever use it in a way that is
independent of the choice of gauge.
Now our criterion is this:
the representation~$\rep$ lifts
to the spin group if and only if $\rho_\rep\in X^*(T)$.

When $G$ is no longer abelian, we can reduce
the spin-lifting problem to the abelian case
using the observation that all the obstructions
to lifting lie in a maximal torus.

\begin{lemma} \label{thm9}
Let $f\colon G\to H$ be a homomorphism
of complex reductive groups,
$T\subseteq G$ a split maximal torus of~$G$,
and $\widetilde H\to H$ an isogeny.
Then $f$ lifts to~$\widetilde H$
if and only if $f|_T$ lifts to~$\widetilde H$.
\end{lemma}

The \namecref{thm9}
is true over much more general bases
than the complex numbers.
Counter to the spirit of SGA~3,
we prove it using the analytic topology on~$G$.

\begin{proof}
Isogenies are covering spaces.
This claim follows from the lifting criterion
for covering spaces \cite[Proposition~1.33]{hatcher02}
together with the surjectivity of
the map $\pi_1(T)\to\pi_1(G)$
\cite[Section~4.6]{bourbaki_lie9}.
\end{proof}

Combining \Cref{thm9} with our analysis
of the abelian case completely solves
the problem of lifting an orthogonal representation
to the spin group.

\begin{theorem} \label{thm12}
Let $G$ be a reductive group,
$T\subseteq G$ a maximal torus,
and $(\rep,V)$ an orthogonal representation.
Then $\rep$ lifts to the spin group
if and only if $\rho_\rep\in X^*(T)$.
\end{theorem}

\subsection{Comparison with Bourbaki}
\label{sec:spin:bourbaki}

Our \Cref{thm12} differs from at least one
work in the canon of Lie groups,
Chapter~9 of Bourbaki's \emph{Groupes et algèbres de Lie}
\cite{bourbaki_lie9}.
I believe there is a small error
in Bourbaki's account of the lifting criterion.
In light of the famous scrupulousness
with which the Bourbaki group prepared their treatises,
a few words are in order to explain the discrepancy.

Chapter~9 of Bourbaki's book studies compact connected Lie groups.
Although this setting is different from ours,
the algebraic setting, there is a standard dictionary 
\cite[Section~VIII.6--7]{serre01}
between the compact and algebraic settings,
and this dictionary gives a comparison
between our work and Bourbaki's.
Exercise~7a of Section~7 of \cite{bourbaki_lie9}
concerns the spin lifting question.
The difference between Bourbaki's answer and ours,
the quantity $\rho_\rep$ defined in \Cref{thm8},
is that we take into account the multiplicity of the weights,
in other words, the dimensions of the weight spaces,
while Bourbaki's analogue of~$\rho_\rep$,
\[
\frac12\sum_{0<\alpha\in\Phi(V)} \alpha,
\]
does not weight the sum by multiplicity.
In this subsection we give an example
explaining why Bourbaki's criterion is incorrect.

Our theory predicts that the double
of an orthogonal representation lifts to the spin group.
Here is an independent proof of this prediction.
It shows that Bourbaki's exercise cannot be correct,
as we explain after the proof.

\begin{lemma}\label{thm10}
Let $\rep\colon G\to\SO(V)$ be an orthogonal representation
of a reductive group~$G$.
Then $\rep\oplus\rep$ lifts to the spin group.
\end{lemma}

\begin{proof}
It suffices to prove the claim in the case
where $\rep$ is the tautological (identity)
orthogonal representation of $G=\SO(V)$.
Consider the following commutative diagram,
in which the left horizontal arrows are diagonal inclusions,
the right horizontal arrows are multiplication maps,
and the vertical arrows are the canonical projection
from the spin group to the special orthogonal group.
\begin{center}
\begin{tikzcd}
\Spin(V) \dar \rar{\Delta} &
\Spin(V\oplus0)\times\Spin(0\oplus V) \dar \rar{\bullet} &
\Spin(V\oplus V) \dar \\
\SO(V) \rar{\Delta} &
\SO(V\oplus0)\times\SO(0\oplus V) \rar{\bullet} &
\SO(V\oplus V).
\end{tikzcd}
\end{center}
The claim amounts to showing that
the bottom horizontal composite arrow
lifts to $\Spin(V\oplus V)$.
Recall that the kernel of $\Spin(V)\to\SO(V)$
is negation in the Clifford algebra,
which we denote by~$-1$.
The claim is equivalent to the statement
that $-1$ lies in the kernel of
the horizontal top composite arrow.
And this statement follows from the fact that
for any isometric embedding $W\to V$ of quadratic spaces,
the induced map $\Spin(W)\to\Spin(V)$ sends
the center to the center.
This fact is a consequence of the Clifford-algebra
definition of the spin group:
the induced spin-group map
is the restriction of the induced Clifford
algebra map $\Cl(W)\to\Cl(V)$,
and this map is the identity on the copy of~$\C$
inside both algebras.
\end{proof}

It is clear that the tautological orthogonal
representation of~$\SO(V)$ does not lift
to the spin group: otherwise,
the spin cover of~$\SO(V)$ would split
and $\Spin(V)$ would be disconnected.
One the other hand, since $\Phi(V)=\Phi(V\oplus V)$
and Bourbaki's criterion is sensitive only
to the set of weights, that criterion would imply,
along with \Cref{thm10},
that the tautological representation does lift.

\subsection{Spin character}
\label{sec:spin:char}
In this subsection we build on our work
from \Cref{sec:spin:crit} to give
a formula for the spin character
of an orthogonal representation $\rep\colon G\to\SO(V)$.

First, let's review the construction
of the universal covering projection
of a complex torus~$T$.
The universal cover of~$T$ can be identified
with the Lie algebra~$\frak t$ and
the universal covering map $\frak t\to T$
is then the exponential map from the theory of Lie groups.
The universal cover is functorial
because formation of Lie algebras is functorial:
the morphism on universal covers
induced by a morphism of tori
is the differential of the morphism
on the Lie algebras.

For our purposes it is more useful,
however, to describe the universal cover $\frak t$
using cocharacter lattices.
The evaluation map $X_*(T)\otimes\C^\times\to T$
is an isomorphism, forming the tensor product over~$\Z$.
It is conventional to use exponential notation
for these tensors, writing
$a^\lambda$ for $\lambda\otimes a$.
Similarly, evaluation of the derivative at~$1$ gives
a canonical isomorphism $X_*(T)\otimes\C\to\frak t$.
In these coordinates, the exponential cover $\frak t\to T$
is simply the map $X_*(T)\otimes\C
\to X_*(T)\otimes\C^\times$ induced by
the exponential function
\[
\lambda\otimes a\mapsto \exp(2\pi ia)^\lambda.
\]
The universal cover map identifies its kernel,
$X_*(T)$, with the fundamental group of~$T$.
More concretely, the identification $X_*(T)\simeq\pi_1(T)$
is restriction of cocharacters
to the unit circle of~$\C^\times$.
In this tensor product model,
functoriality of the universal cover
follows from functoriality of the cocharacter lattice:
a homomorphism $f\colon T\to S$ of tori induces
a map $f_*\colon X_*(T)\to X_*(S)$,
and the induced map $f_\tn{univ}\colon T_\tn{univ}\to S_\tn{univ}$
is simply
\[
f_\tn{univ}\colon \lambda\otimes a\mapsto f_*(\lambda)\otimes a.
\]
Of particular importance are the isogenies~$f\colon T\to S$,
for which $f_\tn{univ}$ relates
two different descriptions of the same universal cover.

We first work out the character $e_\rep$
in the case where $G=T$ is a torus;
the general case follows immediately
from this special case.
Choose a gauge on~$X^*(T)\setminus\{0\}$
so that we may speak of positive and negative characters.
Retain the notation from \cref{sec:spin:crit},
so that $I$ indexes a Witt basis~$(e_i)_{i\in I}$
of~$V$ of weight vectors,
$S$ is the maximal torus of~$\SO(V)$ corresponding to the basis,
and $\widetilde S$ is its double cover in~$\Spin(V)$.
Here $X_*(S)=\ker(\Z^I\to\Z^{I/\pm})$,
a basis for $X_*(S)$ is $(f_i - f_{-i})_{i>0}$
where $(f_i)_{\in I}$ is the standard basis of $\Z^I$,
and $X_*(\widetilde S)$ is the set of elements
of $\tfrac12 X_*(S)$ whose $f_i$-coefficients
are either all integers or all half-integers.
Now consider the following diagram,
where $S'=S$ and the dashed arrow is squaring.
\begin{center}
\begin{tikzcd}
T_\tn{univ} \dar{\rep_\tn{univ}}\arrow{rrr} &&&
T \dar{\rep} \\
S_\tn{univ} \rar &
S' \rar \arrow[rr,dashed,bend left,"2"] &
\widetilde S \rar &
S
\end{tikzcd}
\end{center}
The map $X_*(S')=X_*(S)\to X_*(S)$ induced by squaring
is multiplication by~$2$.
Therefore,
under the identification $S_\tn{univ}=X_*(S)\otimes\C$,
the map $S_\tn{univ}\to X_*(S')\otimes\C^\times=S'$ is
\[
\lambda\otimes a\mapsto \exp(2\pi i a/2)^\lambda
\qquad\lambda\in X_*(S).
\]
At the same time, our work in \Cref{sec:spin:crit}
shows that the map $\rep_*\colon X_*(T)\to X_*(S)$ is
\[
\rep_*\colon \lambda\mapsto\sum_{i>0}
\langle\lambda,\alpha_i\rangle(f_i-f_{-i}).
\]
As $\alpha_{-i} = \alpha_i^{-1}$,
it follows that the map $X_*(T)\to\widetilde S$
is given by the formula
\[
\lambda\mapsto
\bigl(\exp(\pi i\langle\lambda,
\alpha_i\rangle)\bigr)_{i\in I},
\]
where the target element
is interpreted as a coset in~$S$
following the discussion in \Cref{sec:spin:crit}.
Belying the notational complexity of this formula,
every component of the tuple is~$\pm1$.
We know that the image in~$\widetilde S$
of this tuple lies in the center of~$\Spin(V)$,
an order-two subgroup of~$\widetilde S$.
To identify the image as~$+1$ or~$-1$,
we take the product of the elements of the tuple
with positive index~$i$.
The final formula, therefore, is
\[
e_\rep(\lambda)
= \prod_{i>0} \exp\bigl(\pi i\langle\lambda,
\alpha_i\rangle\bigr)
= \exp\bigl(\pi i\langle\lambda,2\rho_\rep\rangle\bigr)
= (-1)^{\langle\lambda,2\rho_\rep\rangle},
\qquad \lambda\in X_*(T)\simeq\pi_1(T).
\]

When $G$ is not a torus, we choose a split maximal torus~$T$
in~$G$ and use the fact that the inclusion $i\colon T\to G$
induces a surjection $\pi_1(T)\to\pi_1(G)$.
A diagram chase shows that
the spin character $e_{\rep\circ i}$
for the restriction of~$\rep$ to~$T$
factors through $e_\rep$.
In this way we reduce to the case where $G$ is a torus.

\begin{theorem} \label{thm24}
Let $G$ be a complex reductive group,
let $T\subseteq G$ be a maximal torus,
and let $\rep\colon G\to\SO(V)$ be an orthogonal representation.
The spin character $e_\rep\colon\pi_1(G)\to\{\pm1\}$
induced by~$\rep$ is given by the formula
\[
e_\rep(\lambda) = (-1)^{\langle\lambda,2\rho_\rep\rangle},
\]
where $\lambda\in X_*(T)\simeq\pi_1(G)$
and $\rho_\rep$ is defined in \eqref{thm8}.
\end{theorem}

\section{Central characters and Weil cohomology} \label{sec:cc}
To apply \Cref{thm4} to prove \Cref{thm1},
we need to interpret the image
of the class $\varphi^*(c_G)$ under the map
\[
e_{\rep,*}: \tn H^2(W_k,\pi_1(\widehat G)) \to
\tn H^2(W_k,\{\pm1\}).
\]
It turns out that $\varphi^*(c_G)$
corresponds to Langlands's central character $\chi_\varphi$
-- conjecturally, the central character
of the $L$-packet of~$\varphi$
-- and that the map $e_{\rep,*}$ corresponds to evaluation
of the character on a certain involution~$z_\rep$.
The goal of this section is
to justify and explain these claims,
and ultimately, in \Cref{thm29}, to show that
$e_{\rep,*}\varphi^*(c_G)=\chi_\varphi(z_\rep)$.

Let $Z$ be the center of the quasi-split
reductive group~$G$.
The main difficulty is to interpret
the cohomology group
\[
\tn H^2(W_k,X^*(Z)).
\]
After reviewing the general setting
in which the cohomology of the Weil group
is computed, in \Cref{sec:cc:wc},
we show in \Cref{sec:cc:hc}
that when $Z$ is connected,
this group is the character group
of the Harish-Chandra subgroup
$Z(k)^\HC$ of~$Z(k)$.
When $\fieldchar k = 0$ and $k$ is nonarchimedean
this identification exists even for $Z$ disconnected,
as we explain in \Cref{app:karpuk}.
Although the center of~$G$
need not be connected in general,
Langlands's definition of~$\chi_\varphi$
permits us to reduce, in \Cref{sec:cc:ei},
to the connected-center case.
We conclude in \Cref{sec:cc:inv}
by defining the involution~$z_\rep$
and proving \Cref{thm29}.

In what follows, we use continuous cohomology
for the Weil group of a nonarchimedean field
and Borel cohomology for the Weil group
of an archimedean field.
The difference is sometimes elided in the notation
to avoid overburdening the reader.

\subsection{Cohomology of the Weil group}
\label{sec:cc:wc}
What kind of group cohomology should
we use to study the Weil group?
Over an archimedean field
the right answer is Borel cohomology,
as \Cref{thm31} shows.
Over a nonarchimedean field,
one%
\footnote{Lichtenbaum remarked
that Borel and continuous cohomology
agree in many cases, in particular,
when the coefficient group is discrete
and countable \cite[Remark~2.2]{lichtenbaum09}.
We will not use this result,
though it would slightly simplify the exposition.}
right answer is continuous cohomology.
Some care is required, however,
because the Weil group is not profinite,
only locally profinite.
For profinite groups we can easily reduce
most foundational problems to the setting of finite groups,
where topology is irrelevant, but this
is no longer the case for locally profinite groups.

Fortunately, Flach has written a nice article
\cite{flach08} that addresses technical concerns
in the continuous cohomology of the Weil group.
Flach first situates this cohomology
in a general topos-theoretic setting,
using theory from SGA~4,
and then shows that this general theory recovers
the usual definition of continuous cohomology
by continuous cochains.
One particularly useful consequence
of his work is the existence of the usual
long exact sequence for any short exact sequence
of topological modules whose quotient map
locally admits continuous sections
\cite[Lemma~6]{flach08}.
Another useful consequence is
the following \namecref{thm23}.

\begin{lemma} \label{thm23}
Let $V$ be a discrete $W_k$-module whose underlying
abelian group is uniquely divisible,
in other words, a $\Q$-vector space.
\begin{enumerate}
\item
If $k$ is nonarchimedean
then $\tn H^i(W_k,V) = 0$ 
for $i\geq 2$.
\item
If $k$ is archimedean then
$\tn H_\tn{Borel}^i(W_k,V) = 0$ 
for $i$ odd.
\end{enumerate}
\end{lemma}

\begin{proof}
First, assume $k$ is nonarchimedean.
There is a Hochschild-Serre spectral sequence
\[
\tn H^i(\Z,\tn H^j(I_k,V))
\implies \tn H^{i+j}(W_k,V)
\]
coming from the short exact sequence
$1 \to I_k \to W_k \to \Z \to 1$.
This spectral sequence is an instance of
a more general spectral sequence that Flach constructed
\cite[Corollary~6]{flach08}.
In the more general sequence,
the group $\tn H^j(I_k,V)$ is replaced
by a sheaf which may or may not be representable.
But since $V$ is discrete and
$W_k$ is locally profinite,
this sheaf is representable
\cite[Proposition~9.2]{flach08}
and there are no technical problems.

Since $V$ is uniquely divisible and $I_k$ is profinite,
$\tn H^j(I_k,V) = 0$ for $j\geq1$
\cite[Proposition~1.6.2]{neukirch_schmidt_wingberg08}.
Since $\Z$ has cohomological dimension one,
$\tn H^i(\Z,-)$ vanishes for $i\geq2$.
So all entries on the starting page of
the spectral sequence vanish except
possibly those in positions
$(0,0)$ and $(1,0)$.

Next, suppose $k$ is archimedean.
If $k=\C$ then $W_\C=\C^\times= S^1\times\R_{>0}$
has classifying space
the infinite-dimensional complex projective space.
Its integral, hence rational, cohomology
is known to be concentrated in even degrees.
If $k=\R$ then use the Hochschild-Serre spectral sequence
in Borel cohomology \cite[Theorem~9]{moore76a}
for the short exact sequence
$1\to\C^\times\to W_\R\to\Gamma\!_\R\to1$
together with the vanishing of $\tn H^{\geq1}(\Gamma\!_\R,-)$
on uniquely divisible groups to show that the map
$\tn H_\tn{Borel}^i(W_\R,V)
\to\tn H_\tn{Borel}^i(\C^\times,V)^{\Gamma\!_\R}$
is an isomorphism for $i\geq2$.
\end{proof}

\subsection{Harish-Chandra subgroup} \label{sec:cc:hc}
Recall that $G$ is a quasi-split reductive $k$-group.
Define the pairing $G(k)\otimes
\Hom_\tn{$k$-gp}(G,\G_{\tn m})\to\R$ by
\[
\langle g,\alpha\rangle = \ord_k(\alpha(g)),
\]
where $\ord_k\colon k^\times\to\R$
is the absolute value when $k$ is archimedean
and the discrete valuation when $k$ is nonarchimedean.
Currying yields a map
\[
\theta_G: G(k)\to \Hom(\Hom_\tn{$k$-gp}(G,\G_{\tn m}),\R).
\]
In particular, if $G=T$ is a torus
then the target of $\theta_G$ is $X_*(T)^{\Gamma\!_k}_\R$.
A character of $G(k)$ is \emph{unramified}
if it is inflated along $\theta_G$.

\begin{definition}
Let $G$ be a $k$-group.
The \emph{Harish-Chandra subgroup} $G(k)^\HC$ of~$G(k)$
is the kernel of the homomorphism~$\theta_G$.
\end{definition}

For example, if $G$ is finite
then $G(k)^\HC = G(k)$ because
roots of unity have valuation zero.
If $G=T$ is a torus
then the group $T(k)^\HC$ is compact,
hence profinite.
The Harish-Chandra subgroup is relevant 
because its character group is exactly
the cohomology group of interest to us,
as we now explain.

\begin{lemma} \label{thm30}
Let $T$ be a $k$-torus.
The image of the composite map
(taken in Borel cohomology if $k$ is archimedean)
\begin{center}
\begin{tikzcd}
\tn H^1(W_k,X^*(T)_\C) \rar &
\tn H^1(W_k,\widehat T) \rar{\tn{LLC}} &
\Hom_\tn{cts}(T(k),\C^\times),
\end{tikzcd}
\end{center}
in which the second map is
the local Langlands correspondence for tori,
is the group of unramified characters of~$T(k)$.
\end{lemma}

\begin{proof}
When $T=\G_{\tn m}$ the claim is clear, and
Shapiro's lemma then proves the claim
for any induced torus.
In general, embed $T$ in an induced torus~$S$.
Using the diagram
\begin{center}
\begin{tikzcd}
\tn H^1(W_k,X^*(S)_\C)
\dar[twoheadrightarrow]\rar &
\tn H^1(W_k,\widehat S) \dar\rar{\tn{LLC}} &
\Hom_\tn{cts}(S(k),\C^\times) \dar \\
\tn H^1(W_k,X^*(T)_\C) \rar &
\tn H^1(W_k,\widehat T) \rar{\tn{LLC}} &
\Hom_\tn{cts}(T(k),\C^\times)
\end{tikzcd}
\end{center}
and the fact that the left vertical arrow
is a surjection by \Cref{thm23},
it suffices to show that
the unramified characters of~$T(k)$
are precisely the restrictions
of the unramified characters of~$S(k)$.
Clearly restriction preserves unramification,
and every unramified character of~$T(k)$
extends to an unramified character of~$S(k)$
because $X_*(T)^{\Gamma\!_k}$
is a summand of $X_*(S)^{\Gamma\!_k}$.
\end{proof}

The universal cover exact sequence for $\widehat T$
expands, by \Cref{thm23}, to the exact sequence
\begin{center}
\begin{tikzcd}
\tn H^1(W_k,X^*(T)_\C) \rar &
\tn H^1(W_k,\widehat T) \rar &
\tn H^2(W_k,X^*(T)) \rar &
0.
\end{tikzcd}
\end{center}
At the same time, there is a canonical identification
\[
\coker\bigl(\Hom(X_*(T)_\R^{\Gamma\!_k},\C^\times)
\to\Hom_\tn{cts}(T(k),\C^\times)\bigr)
\simeq \Hom_\tn{cts}(T(k)^\HC,\C^\times).
\]
Combining these two facts
yields a cohomological description
of the character group of~$T(k)^\HC$.

\begin{corollary}
Let $T$ be a $k$-torus.
The local Langlands correspondence
and universal cover coboundary
furnish an isomorphism
(in Borel cohomology if $k$ is archimedean)
\[
\tn H^2(W_k,X^*(T)) \simeq \Hom_\tn{cts}(T(k)^\HC,\C^\times).
\]
\end{corollary}

\subsection{Evaluation at an involution}
\label{sec:cc:ei}
Let $T$ be a $k$-torus.
Every Galois-equivariant homomorphism
$X^*(T) \to \{\pm1\}$ induces a map
\[
\tn H^2(W_k,X^*(T)) \to \tn H^2(W_k,\{\pm1\}).
\]
The source is the character group of~$T(k)^\HC$
and the target is a group of order two.
It follows by Pontryagin duality
that this map is evaluation of characters
at an involution.
The goal of this subsection is to show that
when $T=\G_{\tn m}$ and the map
$X^*(T)=\Z\to\{\pm1\}$ is nontrivial,
this involution is nontrivial.
Although the claim is quite weak,
it is all that is needed for our computation
of $e_{\rep,*}\varphi^*(c_G)$.

When $k$ is nonarchimedean of characteristic zero
we can describe $\tn H^2(W_k,X^*(T))$
using the cup-product pairing
of Tate duality,
as in \Cref{app:karpuk},
and the claim follows immediately from
the naturality of that pairing.
But in general there is no such
naturality statement compatible with our
coboundary description of $\tn H^2(W_k,X^*(T))$.
Instead, we prove nontriviality
using an argument with long exact sequences
that rests on the following computation.

\begin{lemma} \label{thm27}
Let $A$ be a $W_k$-module which is
finitely generated as an abelian group.
\begin{enumerate}
\item
If $k$ is archimedean then
$\tn H_\tn{Borel}^3(W_k,A)$ is torsionfree.
\item
If $k$ is nonarchimedean then
$\tn H^3(W_k,A)$ is a $p$-group.
\end{enumerate}
\end{lemma}

A (possibly infinite) group is a \emph{$p$-group}
if each of its elements has order a power of~$p$.

\begin{proof}
When $k$ is archimedean,
we can argue as in the proof of \Cref{thm23}.
For $k=\C$ we conclude that
$\tn H_\tn{Borel}^3(W_\C,A)=0$
and for $k=\R$ we conclude,
using the calculation on page~310 of \cite{deligne76},
that $\tn H_\tn{Borel}^3(W_\R,A)$ is
a free $\Z$-module of finite type.

Next, assume $k$ is nonarchimedean.
When $\fieldchar k=0$,
Karpuk proved \cite[Section~3.2]{karpuk13}
the stronger statement
that $\tn H^i(W_k,A) = 0$ for all $i\geq3$.
A slight modification of his argument,
which we outline below, proves the \namecref{thm27}.

The Hochschild-Serre spectral sequence,
mentioned in the proof of \Cref{thm23},
gives the vanishing for $i\geq4$
and can be used to show, for $i=3$,
that $\tn H^2(W_k,P) \simeq \tn H^3(W_k,M)$,
where $P\defeq M_\Q/M$.
Let $P[n]$ denote the $n$-torsion subgroup of~$P$.
The cohomology of $W_k$ in $P$ can be computed
as a direct limit of the cohomologies
in the torsion subgroups:
\[
\tn H^2(W_k,P)
\simeq \varinjlim_n \tn H^2(W_k,P[n]).
\]
Since $P[n]$ decomposes as a direct sum
of its Sylow subgroups,
we can rewrite the direct limit as
\[
\tn H^2(W_k,P)
\simeq \tn H^2(W_k,P[p^\infty])
\oplus \varinjlim_{p\nmid n} \tn H^2(W_k,P[n])
\]
where $P[p^\infty]\defeq\bigcup_{n\geq0} P[p^n]$.
Since the group $P[p^\infty]$ is a $p$-group,
the cohomology group $\tn H^2(W_k,P[p^\infty])$
is a $p$-group as well.
It therefore suffices to show
that the second summand above vanishes.

By Tate duality (see \Cref{thm14}),
if $n$ is prime to~$p$ then
$\tn H^2(W_k,P[n])$ is dual to $\Hom_{W_k}(P[n],\mu)$
where $\mu$ is the group of roots of unity in~$\bar k$.
It follows that the second term above,
$\varinjlim_{p\nmid n} \tn H^2(W_k,P[n])$,
is dual to
\[
\Hom_{W_k}(P/P[p^\infty],\mu).
\]
Since $P[p^\infty]$ is a summand of~$P$,
it suffices to show that $\Hom_{W_k}(P,\mu) = 0$.
Karpuk has an argument for the vanishing
that works just as well when $\fieldchar k > 0$.
\end{proof}

\Cref{thm27} together with a long exact sequence argument
proves our desired nontriviality.

\begin{corollary} \label{thm28}
Let $\ell$ be a finite separable extension of~$k$,
let $M = \Ind_{\ell/k}\Z$,
and let $f\colon M\to\{\pm1\}$ be
the unique nontrivial homomorphism.
If $\fieldchar k\neq 2$ then the induced map
\[
f_*: \tn H^2(W_k,M) \to
\tn H^2(W_k,\{\pm1\})
\]
(taken in Borel cohomology if $k$ is archimedean)
is nontrivial.
\end{corollary}

\begin{proof}
Since $f$ factors through
the augmentation map $M\to\Z$,
it suffices to prove the \namecref{thm28}
in the case where $\ell=k$ and $M=\Z$.
The long exact sequence for
the short exact sequence
$1\to\Z\xto2\Z\to\{\pm1\}\to1$
yields a coboundary map
\[
\tn H^2(W_k,\{\pm1\})\to\tn H^3(W_k,\Z) 
\]
whose cokernel measures
the failure of surjectivity.
But by \Cref{thm27} the coboundary map
must be trivial:
there are no nontrivial homomorphisms
between $\{\pm1\}$ and a torsionfree group
or a $p$-group with $p\neq 2$.
\end{proof}

\subsection{Involutions from representations of
\texorpdfstring{$L$}{L}-groups}
\label{sec:cc:inv}
In this subsection we explain how to construct
from an orthogonal representation~$\rep$ of~${}^LG$
a central involution $z_\rep\in Z(k)$,
generalizing Gross and Reeder's canonical involution.
We then relate this involution
to our description of the character group of~$Z(k)^\HC$.

The involution $z_\rep$ depends only on
the restriction of~$\rep$ to~$\widehat G$.
This restriction is an algebraic representation
of a complex reductive group and can thus be understood
through its weights.
We compute the weights of the representation
with respect to a fixed, Galois-stable maximal torus
$\widehat T$ of~$\widehat G$,
which is dual to a minimal Levi~$T$
of the quasi-split group~$G$.
Since $\rep|_{\widehat G}$ is the restriction
of a representation of the $L$-group,
the multiset of its weights is Galois-stable.
Moreover, each weight can be interpreted
as a coweight of~$T$.

Let $m\colon X_*(T)\to\N$ be
a Galois-invariant multiset of weights.
In our application $m$ will be
the multiset of weights of a representation of~${}^LG$,
but the greater generality is convenient for some proofs.
As in \Cref{sec:spin:crit}, choose a gauge
$X^*(\widehat T)\setminus\{0\} \to\{\pm1\}$.
This time we require the gauge to be Galois-invariant,
and the requirement can be met
because the Galois action preserves
some pinning containing~$\widehat T$.
Define
\[
z_m \defeq \sum_{0<\varpi\in X_*(T)} \varpi(-1)^{m(\varpi)}.
\]
The involution~$z_m$ is independent of the choice of gauge.
For $\rep$ a representation of~${}^LG$ we set
$z_\rep\defeq z_m$ where $m(\varpi)$
is the multiplicity of $\varpi$ in $\rep|_{\widehat T}$.

In this setting, the generalization
of the spin character of \Cref{sec:spin:char}
is the character
$e_m\colon\pi_1(\widehat G)\simeq X^*(T)\to\{\pm1\}$
defined by the formula
\[
e_m(\lambda) = \prod_{0<\varpi\in X_*(T)}
(-1)^{m(\varpi)\langle\lambda,\varpi\rangle}.
\]

\begin{lemma} \label{thm25}
Let $m\colon X_*(T)\to\N$ be a Galois-invariant multiset of weights.
If $G$ has connected center then the composite map
\begin{center}
\begin{tikzcd}
\Hom_\tn{cts}(Z(k)^\HC,\C^\times) \rar{\simeq} &
\tn H^2(W_k,\pi_1(\widehat G)) \rar{e_{m,*}} &
\tn H^2(W_k,\{\pm1\}) \rar{\sgn} & \{\pm1\}
\end{tikzcd}
\end{center}
(taken in Borel cohomology if $k$ is archimedean)
is evaluation at~$z_m$.
\end{lemma}

\begin{proof}
By additivity it suffices to consider the case
where $m$ is a (multiplicity-one) Galois orbit~$\Omega$.
After choosing a gauge, we can factor
$e_m\colon X^*(T)\to\{\pm1\}$ as the composition
\begin{center}
\begin{tikzcd}
X^*(T) \rar &
\Z[\Omega] \rar &
\{\pm1\}
\end{tikzcd}
\end{center}
of the map
\[
\lambda\mapsto\sum_{\varpi\in\Omega}
\langle\lambda,\varpi\rangle[\varpi]
\]
and the mod-two augmentation map $\Z[\Omega]\to\{\pm1\}$.

Apply the functor $\tn H^2(W_k,-)$ to this composition.
By naturality of the coboundary,
the image of the first map is isomorphic to
the restriction map
\[
\Hom_\tn{cts}(T(k)^\HC,\C^\times)
\to \Hom_\tn{cts}(S(k)^\HC,\C^\times)
\]
where $S$ is the $k$-torus with character group~$\Z[\Omega]$.
The image of the second map is a homomorphism
\[
\Hom_\tn{cts}(S(k)^\HC,\C^\times) \to\{\pm1\}
\]
which is nontrivial by \Cref{thm28}.
And by Pontryagin duality this map must be
evaluation at an element of order two.
Since $S(k)\simeq\ell^\times$
where $\ell$ is the fixed field of any element of~$\Omega$,
the group $S(k)$ has a unique element of order two,
namely~$-1$.
The \namecref{thm25} follows.
\end{proof}

Let $\chi_\varphi$ denote Langlands's central character,
the construction of which is summarized
in Borel's Corvallis article
\cite[10.1]{borel_corvallis2},

\begin{theorem} \label{thm29}
Let $\varphi\colon W_k\to{}^LG$ be an $L$-parameter,
let $\rep\colon {}^LG\to\tn O(V)$
be a complex representation of~${}^LG$,
and let $c_G\in\tn H_\tn{Borel}^2({}^LG,\pi_1(\widehat G))$
classify the extension $\widehat G_\tn{univ}\rtimes W_k$.
Then
\[
\sgn(e_{\rep,*}\varphi^*(c_G)) = \chi_\varphi(z_\rep).
\]
\end{theorem}

\begin{proof}
First, assume the center of~$G$ is connected.
Let $\varphi_0\colon W_k\to\widehat G$
denote the $1$-cocycle corresponding to~$\varphi$,
so that $\varphi(w) = \varphi_0(w)w$.
The morphism of short exact sequences
\begin{center}
\begin{tikzcd}
1 \rar &
\pi_1(\widehat G) \rar\dar[rotate=90]{\simeq} &
\widehat G_\tn{univ} \rar\dar &
\widehat G \rar\dar &
1 \\
1 \rar &
X^*(Z) \rar &
\hat{\frak z} \rar &
\widehat Z \rar &
1
\end{tikzcd}
\end{center}
gives rise to a commutative square
\begin{center}
\begin{tikzcd}
\tn H^1(W_k,\widehat G) \rar\dar &
\tn H^1(W_k,\widehat Z) \dar \\
\tn H^2(W_k,\pi_1(\widehat G)) \rar{\simeq} &
\tn H^2(W_k, X^*(Z)).
\end{tikzcd}
\end{center}
Pass $\varphi_0$ around this square.
By \Cref{thm16}, the left arrow
maps $\varphi_0$ to~$\varphi^*(c_G)$.
The top arrow maps $\varphi_0$
to the cocycle representing Langlands's character
$\chi_\varphi\colon Z(k)\to\C^\times$,
and the image of this cocycle under the right arrow
corresponds to the restriction of~$\chi_\varphi$
to~$Z(k)^\HC$.
Hence this restriction corresponds to~$\varphi^*(c_G)$.
The \namecref{thm29} now follows from \Cref{thm25}.

For a general~$G$, whose center need not be connected,
Langlands's construction of~$\chi_\varphi$
starts with a choice of embedding of~$G$
into a group~$G_1$ with connected center.
Let $Z_1$ be the center of~$G_1$.
It turns out that every $L$-parameter
$\varphi\colon W_k\to{}^LG$ lifts to an $L$-parameter
$\varphi_1\colon W_k\to{}^LG_1$.
One then defines $\chi_\varphi$
as the restriction to~$Z(k)$
of the character~$\chi_{\varphi_1}$ of~$Z_1(k)$.
The resulting character $\chi_\varphi$
is independent of the choice of lift
and the choice of~$G_1$.

The projection $\widehat G_1\to\widehat G$
induces the diagram
\begin{center}
\begin{tikzcd}[row sep=small]
&
\pi_1(G_1)
\rar[hookrightarrow] \arrow[dd, "f"]
\dlar[swap]{e_{\rep\circ h}} &
\widehat G_{1,\tn{univ}}\rtimes W_k
\rar[twoheadrightarrow] \arrow[dd, "g"] &
{}^LG_1 \arrow[dd, "h"] &
& (c_1 \defeq c_{G_1}) \\
\{\pm1\} & & & &
W_k \ular[swap]{\varphi_1} \dlar{\varphi} \\
&
\pi_1(G) \ular{e_\rep} \rar[hookrightarrow] &
\widehat G_\tn{univ}\rtimes W_k
\rar[twoheadrightarrow] &
{}^LG & & (c_G)
\end{tikzcd}
\end{center}
with exact rows.
Let $c_1\in\tn H^2({}^LG_1,\pi_1(G_1))$
classify the upper extension.
On the one hand, the existence of~$g$
implies that $h^*(c_G) = f_*(c_1)$,
from which we conclude that
\[
e_{\rep,*}\varphi^*(c_G)
= e_{\rep,*}\varphi_1^*\,h^*(c_G)
= e_{\rep,*}\varphi_1^*\,f_*(c_1)
= e_{\rep\circ h,*}\varphi_1^*(c_1). 
\]
On the other hand, its easy to see
from the definition of~$z_\rep$ that
\[
\chi_\varphi(z_\rep)
= \chi_{\varphi_1}(z_{\rep\circ h}).
\]
These calculations reduce the general
case to the previous one.
\end{proof}

\section{Synthesis} \label{sec:synth}
In this section we complete the proof
of \Cref{thm1} and explain its relationship
to conjectures of Hiraga, Ichino, and Ikeda.

\subsection{Proof of Theorem~A}
Let $G$ be a quasi-split reductive $k$-group,
let $\varphi\colon\WD_k\to{}^LG$ be a tempered $L$-parameter,
and let $\rep\colon{}^LG\to\tn O(V)$
be a complex orthogonal representation.
Since root numbers of orthogonal Weil-Deligne
representations are unaffected by semisimplification,
we replace $\varphi$ by its restriction
$\varphi\colon W_k\to{}^LG$ to the Weil group.
The natural action of the Weil group on~$\widehat G$
lifts to an action on its universal cover.
Using this action, apply \Cref{thm4} to the diagram
\begin{center}
\begin{tikzcd}
1 \rar &
\pi_1(\widehat G) \rar\dar{e_\rep} &
\widehat G_\tn{univ} \rar\dar &
\widehat G \rar\dar{\rep|_{\widehat G}} &
1 \\
1 \rar &
\{\pm1\} \rar &
\Pin(V) \rar &
\tn O(V) \rar &
1 &
(c_\tn{pin}),
\end{tikzcd}
\end{center}
taking $W=W_k$.
Here $c_\tn{pin}\in\tn H_\tn{Borel}^2(\tn O(V),\{\pm1\})$
classifies the (bottom) pin extension and
$c_G\in\tn H_\tn{Borel}^2({}^LG,\pi_1(\widehat G))$
classifies the extension $\widehat G_\tn{univ}\rtimes W_k$.
We conclude from the \namecref{thm4} that
\begin{equation} \label{thm20}
\rep^*(c_\tn{pin}) =
e_{\rep,*}(c_G)\cdot p^*\rep|_{W_k}^*(c_\tn{pin})
\end{equation}
where $p\colon{}^LG \to W$ is the projection.
Hence
\[
\varphi^*\rep^*(c_\tn{pin}) = \varphi^*e_{\rep,*}(c_G)\cdot
\varphi^*p^*\rep|_{W_k}^*(c_\tn{pin}),
\]
valued in $\tn H^2(W_k,\{\pm1\})$
for $k$ nonarchimedean and
in $\tn H_\tn{Borel}^2(W_k,\{\pm1\})$
for $k$ archimedean.
By our formulation of Deligne's theorem, \Cref{thm19},
\[
\frac{\omega(\varphi,\rep)}{\omega(\varphi,\det\rep)}
= \sgn\bigl(\varphi^*\rep^*(c_\tn{pin})\bigr).
\]
At the same time, letting
$\varphi_\tn{prin}\colon\WD_k\to{}^LG$
denote the principal parameter,
the composition $\rep|_{W_k}\circ p\circ\varphi$
is the restriction of~$\rep\circ\varphi_\tn{prin}$
to the Weil group.
Deligne's theorem again implies that
\[
\frac{\omega(\varphi_\tn{prin},\rep)}%
{\omega(\varphi_\tn{prin},\det\rep)}
= \sgn\bigl(\varphi^*p^*\rep|_{W_k}^*(c_\tn{pin})\bigr).
\]
Since $\widehat G$ is connected and $\rep$ is orthogonal,
$\det\rep$ restricts trivially to~$\widehat G$.
Hence 
\[
\omega(\varphi_\tn{prin},\det\rep)
= \omega(\varphi,\det\rep).
\]
All in all, then, \eqref{thm20} simplifies to
\[
\frac{\omega(\varphi,\rep)}%
{\omega(\varphi_\tn{prin},\rep)}
= \sgn\bigl(\varphi^*e_{\rep,*}(c_G)\bigr).
\]
Finally, \Cref{thm29} identifies
the righthand side with $\chi_\varphi(z_\rep)$.

\subsection{Application to Plancherel measure}
Let $G$ be a reductive $k$-group
and let $M$ be a Levi subgroup of~$G$.
From this data we can form
an orthogonal representation~$r_M$ of~${}^LM$
which one might call the relative adjoint
representation for $G\supseteq M$,
namely, the representation
\[
\Lie(\widehat G)/\Lie(Z(\widehat M)^{\Gamma\!_k}).
\]
The nonzero weights of this representation on~$\widehat M$
are the elements of the root system of~$\widehat G$.

Now let $\varphi\colon\WD_k\to{}^LM$
be a tempered discrete $L$-parameter for~$M$
and $\pi$ a representation of~$G(k)$
in the $L$-packet of~$\varphi$.
Hiraga, Ichino, and Ikeda predicted
a relationship between the gamma factor
of the Weil-Deligne representation
$r_M\circ\varphi$ and
the Plancherel measure on the part
of the tempered dual of~$G(k)$
coming from parabolic inductions
of unramified twists of~$\pi$
\cite[Conjecture~1.5]{hiraga_ichino_ikeda08}.
\Cref{thm1} largely computes the root number
of this Weil-Deligne representation:
\[
\omega(\varphi,r_M)
= \omega(\varphi_\tn{prin},r_M)
\cdot \pi(z_\tn{Ad})
\]
where $z_\tn{Ad}$ is Gross and Reeder's
canonical involution for~$G$ (not~$M$).

\appendix
\section{Characters via Weil-Tate duality}
\label{app:karpuk}

Let $T$ be a finite type $k$-group
of multiplicative type.
Our description of $\tn H^2(W_k,X^*(T))$
in \Cref{sec:cc} when $T$ is a torus
brings to mind a related and more classical result,
Tate duality, which in this case
describes $\tn H^2(\Gamma\!_k,X^*(T))$
as the character group of the profinite
completion $T(k)_\tn{pro}$ of~$T(k)$.
This conclusion holds more generally
whenever $T$ is reduced.
By analogy, one would hope that
our description of $\tn H^2(W_k,X^*(T))$
from \Cref{sec:cc:hc}
as the character group of $T(k)^\HC$
could be strengthened by relaxing
the hypothesis that $T$ is a torus
to the hypothesis that $T$ is reduced.

In this appendix we provide partial
evidence, in \Cref{thm14},
for this hope by giving
such a description in two cases,
both for nonarchimedean~$k$:
when $\fieldchar k$ is arbitrary
and $T$ is finite and reduced,
and when $\fieldchar k=0$ and $T$ is arbitrary.
Our starting point is Karpuk's thesis \cite{karpuk13},
which constructs a cup-product pairing
generalizing Tate duality to the Weil group,
provided that $\fieldchar k=0$.
After explaining how this works,
we will have two identifications
of $\tn H^2(W_k,X^*(T))$ with a character group
when $T$ is a torus,
one by coboundary and one by cup product.
We then show that the two identifications
are inverses of each other.

To compare the Harish-Chandra subgroup
with Karpuk's formulation of Weil-Tate duality,
we need a small lemma on (metric) completions.
Let $\ell=\breve k$ be the completion of
the maximal unramified extension of~$k$
and let $\bar\ell$ be a separable closure of~$\ell$.
By Krasner's lemma, $\Gamma_\ell = I_k$.

\begin{lemma} \label{thm15}
An element $a\in\bar\ell$ is algebraic over~$k$
if and only if its Galois orbit is finite.
\end{lemma}

\begin{proof}
The forward implication is clear.
For the reverse implication,
consider the polynomial $f(x)\in\bar\ell[x]$
whose set of roots is the Galois orbit of~$a$.
Since the coefficients of~$f$ are Galois-invariant,
it suffices to show that
$k$ is the set of Galois-invariant elements of~$\bar\ell$.
This claim is a theorem of Tate
\cite[Theorem~1]{tate_driebergen}
which holds also when $\fieldchar k > 0$
\cite[Chapter~XIII, \S5, Lemma~1]{serre79}.
\end{proof}

\begin{theorem} \label{thm14}
Let $T$ be a finite type $k$-group of multiplicative type.
If $\fieldchar k = 0$ or $T$ is finite and reduced
then the cup-product pairing
\[
\tn H^2(W_k,X^*(T))\otimes
\tn H^0\bigl(W_k,\Hom(X^*(T),\cal O_{\bar\ell}^\times)\bigr) \to\Q/\Z
\]
induces a canonical identification
\[
\tn H^2(W_k,X^*(T))\simeq \Hom_\tn{cts}(T(k)^\HC,\C^\times).
\]
\end{theorem}

Here we interpret $\Q/\Z$ as a subgroup of~$\C^\times$
via the exponential map $t\mapsto\exp(2\pi it)$.

\begin{proof}
Let $\bar k$ be the separable closure of~$k$ in~$\bar\ell$.
First, assume $\fieldchar k=0$.
Karpuk showed that the topological groups
$\tn H^2(W_k,X^*(T))$
and $\Hom(X^*(T),\cal O_{\bar\ell}^\times)^{W_k}$
are Pontryagin dual to each other
\cite[Proposition~3.3.5]{karpuk13}. 
The split short exact sequence
for the valuation on $\bar\ell^\times$
gives rise to an exact sequence
describing Karpuk's group as a kernel:
\begin{center}
\begin{tikzcd}
1 \rar &
\Hom(X^*(T),\cal O_{\bar\ell}^\times)^{W_k} \rar &
\Hom(X^*(T),\bar\ell^\times)^{W_k} \rar &
\Hom(X^*(T),\Z)^{W_k}.
\end{tikzcd}
\end{center}
Since the action of $W_k$ on $X^*(T)$ has finite orbits,
any $W_k$-equivariant homomorphism $X^*(T)\to\bar\ell^\times$
factors through~$\bar k^\times$
by \Cref{thm15}.
Hence we may rewrite the sequence as
\begin{center}
\begin{tikzcd}
1 \rar &
\Hom(X^*(T),\cal O_{\bar\ell}^\times)^{W_k} \rar &
T(k) \rar & X_*(T).
\end{tikzcd}
\end{center}
So the lefthand group is the Harish-Chandra subgroup of~$T(k)$.

When $T$ is finite and reduced,
the fact \cite[Lemma~4.2.1]{deligne76} that the inflation map
\[
\tn H^2(\Gamma\!_k,X^*(T)) \to \tn H^2(W_k,X^*(T))
\]
is an isomorphism
reduces the problem to Tate duality for finite modules.
\end{proof}

Our next goal is to compare the description
of \Cref{thm14} with our earlier coboundary description
when $T$ is a torus.
Tate duality is the intermediary in the comparison.
Let $T(k)_\tn{pro}$ denote
the profinite completion of~$T(k)$.
Since $T(k)^\HC$ is compact,
it is a subgroup of $T(k)_\tn{pro}$.

\begin{lemma} \label{thm17}
Let $T$ be a $k$-torus.
The following square commutes,
where the horizontal arrows are
the cup-product pairings
and the vertical arrows are restriction.
\begin{center}
\begin{tikzcd}
\tn H^2(\Gamma\!_k,X^*(T)) \rar{\simeq}
\dar &
\Hom_\tn{cts}(T(k)_\tn{pro},\C^\times)
\dar[twoheadrightarrow] \\
\tn H^2(W_k,X^*(T)) \rar{\simeq} &
\Hom_\tn{cts}(T(k)^\HC,\C^\times).
\end{tikzcd}
\end{center}
\end{lemma}

\begin{proof}
It suffices to observe that
the subgroup inclusion $T(k)^\HC\into T(k)_\tn{pro}$
can be identified with the map
\[
\tn H^0(W_k,\Hom(X^*(T),\cal O_{\bar\ell}^\times))
\to\tn H^0(\Gamma\!_k,\Hom(X^*(T),\bar\ell^\times))
\]
induced by inclusion
$\cal O_{\bar\ell}^\times \into\bar\ell^\times$.
\Cref{thm15} identifies the righthand $\tn H^0$ group
with $T(k)$.
\end{proof}

Tate duality describes $\tn H^2(\Gamma\!_k,X^*(T))$
as the character group of~$T(k)_\tn{pro}$.
Moreover, the extension
of Artin reciprocity to tori
describes $\tn H^1(\Gamma\!_k,\widehat T)$
as the same character group.
The two comparisons are related
by the coboundary map
for the exponential exact sequence
\begin{center}
\begin{tikzcd}
1 \rar &
X^*(T) \rar &
X^*(T)_\C \rar &
\widehat T \rar & 1.
\end{tikzcd}
\end{center}
There is a twist, however:
the composite isomorphism below
is inversion \cite[Section~8.2]{gross_reeder10}.
\[
\Hom_\tn{cts}(T(k)_\tn{pro})
\simeq \tn H^1(\Gamma\!_k,\widehat T)
\to \tn H^2(\Gamma\!_k,X^*(T))
\simeq \Hom_\tn{cts}(T(k)_\tn{pro})
\]

\begin{lemma} \label{thm18}
Let $T$ be a $k$-torus.
The following square anticommutes, where
the top arrow is the (profinitely completed)
Langlands correspondence,
the bottom arrow is from \Cref{thm14},
the left arrow is the exponential coboundary,
and the right arrow is restriction.
\begin{center}
\begin{tikzcd}
\tn H^1(W_k,\widehat T) \rar{\simeq}
\dar &
\Hom_\tn{cts}(T(k),\C^\times)
	\dar[twoheadrightarrow] \\
\tn H^2(W_k,X^*(T)) \rar{\simeq} &
\Hom_\tn{cts}(T(k)^\HC,\C^\times).
\end{tikzcd}
\end{center}
\end{lemma}

\begin{proof}
Consider the following extension of the given diagram
in which the left square commutes by
the compatibility of inflation and coboundary.
\begin{center}
\begin{tikzcd}
\tn H^1(\Gamma\!_k,\widehat T)
	\rar[hookrightarrow]\dar{\simeq} &
\tn H^1(W_k,\widehat T) \rar{\simeq}\dar &
\Hom_\tn{cts}(T(k),\C^\times)
	\dar[twoheadrightarrow] \\
\tn H^2(\Gamma\!_k,X^*(T))
	\rar[twoheadrightarrow] &
\tn H^2(W_k,X^*(T)) \rar &
\Hom_\tn{cts}(T(k)^\HC,\C^\times).
\end{tikzcd}
\end{center}
By \Cref{thm17} and the compatibility
between Tate duality and
the local Langlands correspondence,
the right square becomes commutative
after restriction along
the upper left arrow.
We'll use this to deduce commutativity
of the right square by a diagram chase.

The bottom left arrow
is surjective by \Cref{thm17}
and exactness of Pontryagin duality.
Starting with an element $c\in\tn H^1(W_k,\widehat T)$,
move it clockwise around the left square,
choosing a lift along the bottom left arrow.
We thereby produce an element
$c'\in\tn H^1(W_k,\widehat T)$
in the image of $\tn H^1(\Gamma\!_k,\widehat T)$.
The elements $c$ and~$c'$ have the same image
in~$\tn H^2(W_k,X^*(T))$,
hence differ by an element in the kernel
of the middle vertical arrow,
or in other words, by an element
in the image of $\tn H^1(W_k,X^*(T)_\C)$.
To prove commutativity of the right square,
it therefore suffices to show that
the following composition
$f_T\colon\tn H^1(W_k,X^*(T))
\to\Hom_\tn{cts}(T(k)^\HC,\C^\times)$
vanishes:
\begin{center}
\begin{tikzcd}
\tn H^1(W_k,X^*(T)_\C) \rar &
\tn H^1(W_k,\widehat T) \rar{\simeq} &
\Hom_\tn{cts}(T(k),\C^\times) \rar[twoheadrightarrow] &
\Hom_\tn{cts}(T(k)^\HC,\C^\times).
\end{tikzcd}
\end{center}

For this, we use the same proof strategy
as in the proof of \Cref{thm30}.
When $T=\G_{\tn m}$, the image of the first map
in $\Hom_\tn{cts}(k^\times,\C^\times)$
is the group of unramified characters
and the claim is clear.
It follows from Shapiro's lemma
that $f_T$ vanishes for any torus~$T$
that is induced, in other words,
a product of Weil restrictions of split tori.
For a general torus~$T$,
choose an embedding $T\into S$ into
an induced torus~$S$
and let $R$ denote the cokernel, a third torus.
This embedding yields the commutative square
\begin{center}
\begin{tikzcd}
\tn H^1(W_k,X^*(S)_\C) \rar{f_S}\dar &
\Hom_\tn{cts}(S(k)^\HC,\C^\times) \dar \\
\tn H^1(W_k,X^*(T)_\C) \rar{f_T} &
\Hom_\tn{cts}(T(k)^\HC,\C^\times).
\end{tikzcd}
\end{center}
Since $\tn H^2(W_k,X^*(R)_\C) = 0$
by \Cref{thm23},
the left arrow is surjective.
We can thus deduce the vanishing of~$f_T$
from the vanishing of~$f_S$.
\end{proof}

\bibliography{orntp.bib}
\bibliographystyle{amsalpha}

\end{document}